\documentclass{amsart}
\usepackage{amscd,amsmath,xypic,amssymb,combelow,tikz-cd,etoolbox,calligra,mathrsfs,enumitem,mathtools}

\makeatletter
\patchcmd{\@settitle}{\uppercasenonmath\@title}{}{}{}
\makeatother

\xyoption{all}
\CompileMatrices

\makeatletter
\@addtoreset{equation}{section}
\makeatother

\newtheorem{theorem}[subsection]{Theorem}

\newtheorem{proposition}[subsection]{Proposition}
\newtheorem{lemma}[subsection]{Lemma}

\newtheorem{definition}[subsection]{Definition}

\newtheorem{claim}[subsection]{Claim}

\newtheorem{remark}[subsection]{Remark}

\def\loccitt{\emph{loc. cit.}}
\def\loccit{\emph{loc. cit. }}

\def\fS{{\mathfrak{S}}}

\def\fZ{{\mathfrak{Z}}}

\def\BA{{\mathbb{A}}}
\def\BC{{\mathbb{C}}}

\def\BN{{\mathbb{N}}}
\def\BF{{\mathbb{F}}}
\def\BP{{\mathbb{P}}}
\def\BR{{\mathbb{R}}}
\def\BQ{{\mathbb{Q}}}
\def\BZ{{\mathbb{Z}}}

\def\CA{{\mathcal{A}}}

\def\CE{{\mathcal{E}}}
\def\CF{{\mathcal{F}}}

\def\CK{{\mathcal{K}}}
\def\CL{{\mathcal{L}}}
\def\CM{{\mathcal{M}}}

\def\CO{{\mathcal{O}}}

\def\CS{{\mathcal{S}}}

\def\CU{{\mathcal{U}}}

\def\Hom{\textrm{Hom}}

\def\and{\textrm{ }\&\textrm{ }}

\def\esym{\emph{Sym}}
\def\sym{\textrm{Sym}}

\def\sym{\textrm{Sym}}

\def\A{{\CA}}

\def\la{{\lambda}}

\def\bsq{{\blacksquare}}

\def\bla{{\boldsymbol{\la}}}
\def\bmu{{\boldsymbol{\mu}}}

\def\GA{\Gamma_{\BA^2}}
\def\GS{\Gamma_S}

\def\LA{\Lambda_{\BA^2}}
\def\LS{\Lambda_S}

\def\LSS{\Lambda_S \times K(S)}
\def\LSSS{\Lambda_S \times K(S \times S)}

\def\Ar{{\CA^{(r)}}}
\def\b1{\boldsymbol{1}}

\begin{document}

\title[The universal sheaf as an operator]{\Large{\textbf{The universal sheaf as an operator}}}

\author[Andrei Negu\cb t]{Andrei Negu\cb t}
\address{MIT, Department of Mathematics, Cambridge, MA, USA}
\address{Simion Stoilow Institute of Mathematics, Bucharest, Romania}
\email{andrei.negut@@gmail.com}

\maketitle

\begin{abstract} We compute the universal sheaf of moduli spaces $\CM$ of sheaves on a surface $S$, as an operator $\Lambda = \{\text{symmetric polynomials}\} \rightarrow K(\CM)$, thus generalizing the viewpoint of \cite{CNO} to arbitrary rank and general smooth surfaces.

\end{abstract}

\section{Introduction} 

\noindent Fix a natural number $r$. The moduli space $\CM^f$ of rank $r$ framed sheaves on the plane is an algebro-geometric incarnation of the instanton moduli space that gives rise to supersymmetric ${\mathcal{N}}=2$ $U(r)$--gauge theory on $\BR^4$ in the $\Omega$--background. In \cite{Nek}, the partition function of this theory was expressed in terms of equivariant integrals over $\CM^f$. The present note is concerned with the deformation from cohomology to $K$--theory (over $\BQ$), which corresponds to supersymmetric ${\mathcal{N}}=1$ $U(r)$--gauge theory on $\BR^4 \times S^1$. In this setting, \cite{CNO} considered the universal sheaf:
$$
\xymatrix{\CU \ar@{.>}[d] \\
\CM^f \times \BA^2}
$$
and its exterior powers $\CU_1 \otimes ... \otimes \CU_k$ on $\CM^f \times \BA^{2k}$, where $\CU_i$ denotes the pull-back of $\CU$ from the $i$--th factor of $\BA^{2k} = (\BA^2)^k$. These exterior powers yield operators: 
\begin{equation}
\label{eqn:operator def}
K_T(\CM^f) \xrightarrow{V}  \LA = \bigoplus_{k=0}^\infty K_{\BC^* \times \BC^* \times \fS(k)} (\BA^{2k})
\end{equation}
(the action $T \curvearrowright \CM^f$ is explained in Subsection \ref{sub:torus}, the action $\BC^* \times \BC^* \curvearrowright \BA^2$ is the usual scaling, and the symmetric group $\fS(k)$ acts on $\BA^{2k} = (\BA^2)^k$ by permutations) given by:
\begin{equation}
\label{eqn:operator}
V = \bigoplus_{k=0}^\infty \rho_{k*} \Big(\CU_1 \otimes ... \otimes \CU_k \otimes \pi_{k}^* \Big)
\end{equation}
The maps in \eqref{eqn:operator} are the natural projection maps:
\begin{equation}
\label{eqn:diag 0}
\xymatrix{ & \CM^f \times \BA^{2k}   \ar[ld]_-{\pi_{k}} \ar[rd]^-{\rho_{k}} & \\
\CM^f & & \BA^{2k}}
\end{equation}
As shown in \loccitt, the operator $W(m) : K_T(\CM^f) \rightarrow K_T(\CM^f)$ that encodes the contribution of bifundamental matter to the gauge theory at hand factors as:
\begin{equation}
\label{eqn:ext}
W(m) = V^* \cdot m^{\deg} \cdot V
\end{equation}
(up to a renormalization that will not concern us in the present paper) where $\deg :$ $\LA \rightarrow \LA$ is the operator which scales the $k$--th direct summand in \eqref{eqn:operator def} by $k$. \\

\noindent In \cite{CNO}, from whom we borrowed both the main construction and the title of the present paper, the authors compute the $r=1$ case of $V$ as follows: the moduli space $\CM^f|_{r = 1}$ is isomorphic to the Hilbert scheme of points on $\BA^2$, and its $K$--theory is naturally identified with $\LA$ (see \cite{BKR, H}). Then \cite{CNO} computes $V$ as an explicit exponential in the usual bosonic realization of $\LA$, times the famous $\nabla$ operator (see \cite{BG}). The resulting formula for $V$ yields a geometric incarnation of a combinatorial identity from \cite{GHT}, and implies the formula for $\Phi_m$ computed in \cite{CO}. \\

\noindent In the present paper, we take a somewhat different route toward computing the operator $V$, for general $r$. We recall the actions of the \underline{elliptic Hall algebra} $\A$ on the domain and target of the map \eqref{eqn:operator def}, which were studied in \cite{FHHSY} and \cite{FT, SV 1}:
\begin{equation}
\label{eqn:action}
\LA \ \stackrel{\Psi}\curvearrowleft \ \A \ \stackrel{\Phi}\curvearrowright \  K_T(\CM^f)
\end{equation}
(we refer to \cite{K-theory} for an overview of our viewpoint on these actions). We will recall the definition of the elliptic Hall algebra $\A$ in Subsection \ref{sub:enter a}, and in Subsection \ref{sub:more a} we will introduce the subalgebra:
\begin{equation}
\label{eqn:half 0}
\Ar \subset \A
\end{equation}
Intuitively, $\Ar$ is half of $\A$ with respect to a certain triangular decomposition. Consider the following modification of the diagram \eqref{eqn:diag 0}:
$$
\xymatrix{\CM^f \times \{\text{origin}\} \ar@{^{(}->}[r]^-\iota \ar@{=}[d] & \CM^f \times \BA^{2k}   \ar[ld]_-{\pi_{k}} \ar[rd]^-{\rho_{k}} & \\
\CM^f & & \BA^{2k}}
$$	
from which it is easy to see that:
\begin{equation}
\label{eqn:gamma}
\GA = \bigoplus_{k=0}^\infty \iota^*\Big(\CU_1 \otimes ... \otimes \CU_k \otimes \rho_{k}^* \Big)^{\fS(k)} : \LA \longrightarrow K_T(\CM^f)
\end{equation}
is given by $\GA = V^* \circ [(1-q_1)(1-q_2)]^{\deg}$. Thus, we will compute $\GA$ instead of $V$. \\

\begin{theorem}
\label{thm:main}

For any $r$, the operator $\GA$ commutes with the $\Ar$--action:
\begin{equation}
\label{eqn:main}
\xymatrix{\LA \ar[d]_-{\Psi(a)} \ar[r]^-\GA & K_T(\CM^f) \ar[d]^-{\Phi(a)} \\
\LA \ar[r]^-\GA & K_T(\CM^f)}
\end{equation}
for all $a \in \Ar$. Moreover, after localization to $\emph{Frac}(\emph{Rep}_T)$, the operator $\GA$ is uniquely determined by the commutativity of diagram \eqref{eqn:main}. \\

\end{theorem}

\noindent The point of view in Theorem \ref{thm:main}, namely of determining an operator through its commutation with an algebra action rather than through explicit formulas, was used in \cite{W, AGT} to compute the operator \eqref{eqn:ext}. However, the strength of this approach is that it allows us to generalize from the local situation of the moduli spaces $\CM^f$ of framed sheaves on the affine plane to the global case of moduli spaces of stable sheaves $\CM^s$ on a general smooth projective surface $S$. We will review the necessary setup in Section \ref{sec:surfaces} (including Assumptions A and S, subject to which we make all the following claims), but the idea is to consider the operator: 
\begin{equation}
\label{eqn:operator def surface}
\LS = \bigoplus_{k=0}^\infty K_{\fS(k)} (S^k) \xrightarrow{\GS} K(\CM^s)
\end{equation}
explicitly given by:
\begin{equation}
\label{eqn:operator surface}
\GS = \bigoplus_{k=0}^\infty \pi_{k*} \Big(\CU_1 \otimes ... \otimes \CU_k \otimes \rho_{k}^* \Big)^{\fS(k)} 
\end{equation}
where the notions in the right-hand side of \eqref{eqn:operator surface} are defined just like their counterparts in \eqref{eqn:operator} (Assumption A ensures that there exists a universal sheaf $\CU$ on $\CM^s \times S$, and we fix such a sheaf throughout the present paper). As for the analogues of the action \eqref{eqn:action} for a general surface, the former was worked out in \cite{Hecke}:
\begin{equation}
\label{eqn:phi intro}
\A \xrightarrow{\Phi} \Hom(K(\CM^s) , K(\CM^s \times S))
\end{equation}
and provided a generalization of the classical Heisenberg algebra action on the cohomology groups of Hilbert schemes (\cite{G, Nak}). We will also provide an analogue:
\begin{equation}
\label{eqn:psi intro}
\A \xrightarrow{\Psi} \Hom(\LS , \LSS)
\end{equation}
of the second action from \eqref{eqn:action}, where by a slight abuse of notation, we write:
$$
\LSS = \bigoplus_{k=0}^\infty K_{\fS(k)} (S^k \times S)
$$
(the symmetric group $\fS(k)$ only acts on $S^k$). Then our main result for a smooth surface $S$, subject to the assumptions in Subsection \ref{sub:assumption}, is the following: \\

\begin{theorem}
\label{thm:surface}
	
For any $r$, the operator \eqref{eqn:operator surface} commutes with the $\Ar$--action:
\begin{equation}
\label{eqn:main surface}
\xymatrix{\LS \ar[d]_-{\Psi(a)} \ar[r]^-{\GS} & K(\CM^s) \ar[d]^-{\Phi(a)} \\
\LSS \ar[r]^-{\GS} & K(\CM^s \times S)}
\end{equation}
for all $a \in \Ar$. \\
	
\end{theorem}

\noindent As an $\Ar$--module, $K_T(\CM^f)$ is generated by a single element, namely the fundamental class of the component parametrizing framed sheaves with $c_2 = 0$ (this plays an important role in the uniqueness statement of Theorem \ref{thm:main}). Meanwhile, we will show in Proposition \ref{prop:generators} that $K(\CM^s)$ has countably many generators, namely the fundamental classes $\b1_d$ of the components $\CM^s_d \subset \CM^s$ parametrizing stable sheaves with $c_2 = d$. In our language, we have $\Gamma_S(1) = \prod_{d \in \BZ} \b1_d$, where $1 \in K_{\fS(0)}(S^0) \cong \BQ$. \\

\noindent I would like to thank Andrei Okounkov for teaching me much of the beautiful mathematics contained in this note. I gratefully acknowledge NSF grants DMS--1760264 and DMS--1845034, as well as support from the Alfred P. Sloan Foundation. \\

\section{The case of the affine plane}
\label{sec:affine}

\subsection{} 
\label{sub:frobenius}

Even before dealing with $\BA^2$, let us discuss the situation of $\fS(k)$--equivariant coherent sheaves on a point $\circ$, which is just another word for finite-dimensional $\fS(k)$--modules. We have the well-known Frobenius character isomorphism:
\begin{equation}
\label{eqn:degree k}
K_{\fS(k)} (\circ) \cong \Big\{ \text{degree }k \text{ symmetric polynomials in } x_1,x_2,... \Big\}
\end{equation}
given as a sum over partitions $\lambda = (\lambda_1 \geq ... \geq \lambda_t)$ of size $k$ by the formula:
\begin{equation}
\label{eqn:frobenius}
M \mapsto \sum_{|\lambda| = k} \frac {p_\lambda}{z_\lambda} \cdot \text{Tr}_M(\omega_\lambda)
\end{equation}
where we let $\omega_\lambda \in \fS(k)$ be any permutation of cycle type $\lambda$, and define:
\begin{equation}
\label{eqn:basis}
p_{\lambda} = p_{\lambda_1} \dots p_{\lambda_t} 
\end{equation}
with $p_n = x_1^n+x_2^n+...$ being the \underline{power sum functions}. In \eqref{eqn:frobenius}, we let $z_\lambda = \lambda! \prod_i \lambda_i$, where $\lambda!$ is the product of factorials of the number of times each integer appears in $\lambda$. It is useful to take the direct sum of \eqref{eqn:degree k} over all $k \in \BN \sqcup 0$, and obtain:
\begin{equation}
\label{eqn:degree any}
\Lambda:= \bigoplus_{k=0}^\infty K_{\fS(k)} (\circ) \cong \Big\{\text{symmetric polynomials in } x_1,x_2,... \Big\}
\end{equation}
This is beneficial because $\Lambda$ is manifestly a commutative ring, in fact the polynomial ring generated by $p_1,p_2,...$. In terms of representations of $\fS(k)$, the operation of multiplication by power sum functions corresponds to parabolic induction:
\begin{equation}
\label{eqn:induction}
K_{\fS(l)} (\circ) \xrightarrow{p_k} K_{\fS(k+l)} (\circ), \qquad M \mapsto \text{Ind}_{\fS(k) \times \fS(l)}^{\fS(k+l)} (p_k \boxtimes M)
\end{equation}
However, the ring $\Lambda$ is also endowed with a symmetric pairing, determined by $\langle p_\lambda, p_\mu \rangle = \delta_\lambda^\mu z_\lambda$, or, in the language of finite-dimensional $\fS(k)$--modules:
\begin{equation}
\label{eqn:pairing}
\langle M, M' \rangle = \text{dim } \text{Hom}_{\fS(k)}(M,M')
\end{equation}
With this in mind, Frobenius reciprocity states that the adjoints of the operators \eqref{eqn:induction} are the parabolic restriction operators:
\begin{equation}
\label{eqn:restriction}
K_{\fS(k+l)} (\circ) \xrightarrow{p^\dagger_k} K_{\fS(l)} (\circ), \qquad M \mapsto \text{Hom}_{\fS(k)} \left(p_k, \text{Res}^{\fS(k+l)}_{\fS(k) \times \fS(l)} (M) \right)
\end{equation}
A reformulation of the main result of \cite{Ges} is the following: \\

\begin{theorem}
\label{thm:classic}

The operators $p_k, p_k^\dagger : \Lambda \rightarrow \Lambda$ satisfy the relations:
$$
[p_k^\dagger, p_l] = k \delta_k^l \cdot \emph{Id}
$$
for all $k,l \in \BN$, as well as the obvious relations $[p_k,p_l] = [p_k^\dagger, p_l^\dagger] = 0$. \\

\end{theorem}

\subsection{} 
\label{sub:a2}

We will follow the presentation of \cite{CNO} in the present Subsection, and we will recycle the notation used in the previous Subsection. We will consider $\BA^2$ with the standard action of $\BC^* \times \BC^*$ that dilates the coordinate axes, and then the induced action of $\BC^* \times \BC^*$ on $\BA^{2k} = (\BA^2)^k$ commutes with the action of $\fS(k)$ that permutes the factors. Then we will consider the analogue of \eqref{eqn:degree any}:
$$
\LA = \bigoplus_{k=0}^\infty K_{\BC^* \times \BC^* \times \fS(k)} (\BA^{2k})
$$
If we let $q_1$ and $q_2$ denote the elementary characters of $\BC^* \times \BC^*$, then the inclusion of the origin $\circ \hookrightarrow \BA^2$ induces a map:
$$
\Lambda \otimes_{\BQ} \BQ[q_1^{\pm 1}, q_2^{\pm 1}] \rightarrow \LA
$$
which sends a finite-dimensional $\fS(k)$--module to the same module supported at the origin $\circ^k \hookrightarrow \BA^{2k}$. With this in mind, we may consider the following elements:
$$
\left[ p_k \otimes \CO_{\circ^k} \right] \in  K_{\BC^* \times \BC^* \times \fS(k)} (\BA^{2k})
$$ 
(the skyscraper sheaf at the origin tensored with the $\fS(k)$--character $p_k$) which induce the following analogues of the operators \eqref{eqn:induction} and \eqref{eqn:restriction}:
\begin{equation}
\label{eqn:ind res}
K_{\BC^* \times \BC^* \times \fS(l)} (\BA^{2l}) \xrightleftharpoons[p_k^\dagger]{p_k} K_{\BC^* \times \BC^* \times \fS(k+l)} (\BA^{2(k+l)})
\end{equation}
given by:
\begin{align} 
&p_k(M) = \text{Ind}_{\fS(k) \times \fS(l)}^{\fS(k+l)} \Big( \underbrace{\left[ p_k \otimes \CO_{\circ^k} \right]}_{\text{sheaf on }\BA^{2k}} \ \boxtimes \underbrace{M}_{\text{sheaf on }\BA^{2l}} \Big) \label{eqn:ind} \\
&p_k^\dagger(M) = \text{Hom}_{\fS(k)} \left(p_k, \text{Res}^{\fS(l)}_{\fS(k) \times \fS(l-k)} (M) \Big|_{\circ^k \times \BA^l }\right) \label{eqn:res}
\end{align}
It is easy to see that the operators \eqref{eqn:ind} and \eqref{eqn:res} are adjoint with respect to the pairing on $\LA$ given by formula \eqref{eqn:pairing}, with the caveat that the symbol ``dim Hom" must be understood to mean the $\BC^* \times \BC^*$ character of the space of $\fS(k)$--equivariant global homomorphisms over $\BA^{2k}$. With respect to this pairing, we have:
\begin{equation}
\label{eqn:pairing plane}
\langle p_\lambda, p_\mu \rangle_{\BA^2} = \delta_\lambda^\mu z_\lambda \prod_{i=1}^t \left[ (1-q_1^{\lambda_i})(1-q_2^{\lambda_i}) \right] 
\end{equation}
for any $\lambda = (\lambda_1 \geq ... \geq \lambda_t)$. The natural analogue of Theorem \ref{thm:classic} reads: \\

\begin{proposition}
\label{prop:classic plane}
	
The operators $p_k, p_k^\dagger : \LA \rightarrow \LA$ satisfy the relations:
\begin{equation}
\label{eqn:classic heis}
[p_k^\dagger, p_l] = k \delta_k^l (1-q_1^{k})(1-q_2^{k})\cdot \emph{Id}
\end{equation}
for all $k,l \in \BN$, as well as the obvious relations $[p_k,p_l] = [p_k^\dagger, p_l^\dagger] = 0$. \\
	
\end{proposition}

\begin{proof} See \cite{CNO}, although the proof is analogous to that of Proposition \ref{prop:classic surface}.
	
\end{proof} 

\subsection{} 
\label{sub:plethysm}

Two very important classes of symmetric polynomials are the \underline{elementary} and \underline{complete} ones, whose generating series are given by:
\begin{align*}
&\sum_{k=0}^\infty \frac {h_k}{z^k}  = \exp\left[\sum_{k=1}^\infty \frac {p_k}{k z^k} \right] = \prod_{i=1}^\infty \left(1-\frac {x_i}{z} \right)^{-1} \\
&\sum_{k=0}^\infty \frac {e_k}{(-z)^k} = \exp\left[- \sum_{k=1}^\infty \frac {p_k}{k z^k} \right]  = \prod_{i=1}^\infty \left(1-\frac {x_i}{z} \right)
\end{align*}
As elements of $\Lambda$ and $\LA$ (i.e. as $\fS(k)$--modules or $\fS(k)$--modules supported at the origin of $\BA^{2k}$, respectively), the symmetric polynomials $h_k$ and $e_k$ correspond to the trivial and sign one-dimensional representation, respectively.  Let:
$$
h_k^\dagger, e_k^\dagger : \LA \rightarrow \LA 
$$
denote the adjoints, with respect to the pairing \eqref{eqn:pairing plane}, of the operators of multiplication by $h_k$ and $e_k$ (respectively). Clearly, we have:
\begin{align*} 
&\sum_{k=0}^\infty h_k^\dagger z^k  = \exp\left[\sum_{k=1}^\infty \frac {p^\dagger_k z^k}k \right] \\
&\sum_{k=0}^\infty e_k^\dagger (-z)^k  = \exp\left[-\sum_{k=1}^\infty \frac {p^\dagger_k z^k}k \right] 
\end{align*} 
The following computations are simple consequences of \eqref{eqn:classic heis}:
\begin{align}
&\exp\left[-\sum_{k=1}^\infty \frac {p^\dagger_k z^k}k \right] \exp\left[\sum_{k=1}^\infty \frac {p_kw^{-k}}{k} \right] = \exp\left[\sum_{k=1}^\infty \frac {p_kw^{-k}}{k} \right] \exp\left[-\sum_{k=1}^\infty \frac {p^\dagger_k z^k}k \right]  \zeta \left( \frac zw \right)^{-1} \label{eqn:computation 1}  \\
&\exp\left[\sum_{k=1}^\infty \frac {p^\dagger_k z^k}k \right] \exp\left[- \sum_{k=1}^\infty \frac {p_k w^{-k}}{kq^k} \right] = \exp\left[- \sum_{k=1}^\infty \frac {p_k w^{-k}}{k q^k} \right]  \exp\left[\sum_{k=1}^\infty \frac {p^\dagger_k z^k}k \right]  \zeta \left( \frac wz \right)^{-1} \label{eqn:computation 2} 
\end{align}
where we let $q = q_1q_2$ and write:
\begin{equation}
\label{eqn:zeta} 
\zeta(x) = \frac {(1-xq_1)(1-xq_2)}{(1-x)(1-xq)} = \exp \left[ \sum_{k=1}^\infty \frac {x^k}k \cdot (1-q_1^k)(1-q_2^k) \right]
\end{equation}
Note that $\zeta(x) = \zeta\left(\frac 1{xq} \right)$. \\

\subsection{} 
\label{sub:enter a}

Consider the following half planes in $\BZ^2$:
$$
\BZ_+^2 = \{(n,m) \in \BZ^2 \text{ s.t. } n>0 \text{ or } n=0,m>0\}
$$
$$
\BZ_-^2 = \{(n,m) \in \BZ^2 \text{ s.t. } n<0 \text{ or } n=0,m<0\}
$$
The following is a model for the Hall algebra of the category of coherent sheaves over an elliptic curve, as defined in \cite{BS} (although we follow the normalization of \cite{W}). \\ 

\begin{definition}
\label{def:eha}

Consider the algebra:
$$
\CA_{\emph{loc}} = \BQ(q_1,q_2) \Big \langle P_{n,m}, c_1^{\pm 1}, c_2^{\pm 1} \Big \rangle_{(n,m) \in \BZ^2 \backslash (0,0)} \Big /^{c_1, c_2 \text{ central, and}}_{\text{relations \eqref{eqn:relation 1}, \eqref{eqn:relation 2}}}
$$
where we impose the following relations. The first of these is:
\begin{equation}
\label{eqn:relation 1}
[P_{n,m}, P_{n',m'}] = \delta_{n+n'}^0 \frac {d(1-q_1^d)(1-q_2^d)}{q^{-d} - 1} \left(1 - c_1^{-n} c_2^{-m} \right)
\end{equation}
if $nm'=n'm$ and $(n,m) \in \BZ_+^2$, with $d = \gcd(m,n)$. The second relation states that whenever $nm'>n'm$ and the triangle with vertices $(0,0), (n,m), (n+n',m+m')$ contains no lattice points inside nor on one of the edges, then we have the relation:
\begin{equation}
\label{eqn:relation 2}
[P_{n,m}, P_{n',m'}] = \frac {(1-q_1^d)(1-q_2^d)}{q^{-1} - 1} Q_{n+n',m+m'}
\end{equation}
$$
\cdot \ \begin{cases}
c_1^n c_2^m & \text{if } (n,m) \in \BZ_-^2, (n',m') \in \BZ_+^2, (n+n',m+m') \in \BZ_+^2 \\
c_1^{-n'} c_2^{-m'} & \text{if } (n,m) \in \BZ_-^2, (n',m') \in \BZ_+^2, (n+n',m+m') \in \BZ_-^2 \\
1 & \text{otherwise}
\end{cases}
$$
where $d = \gcd(n,m)\gcd(n',m')$ (by the assumption on the triangle, we note that at most one of the pairs $(n,m), (n',m'), (n+n',m+m')$ can fail to be coprime), and:
\begin{equation}
\label{eqn:qmn}
\sum_{k=0}^{\infty} Q_{ka,kb} \cdot x^k = \exp \left[ \sum_{k=1}^\infty \frac {P_{ka,kb}}k \cdot x^k \left(1 - q^{-k} \right) \right] 
\end{equation}
for all coprime integers $a,b$. Note that $Q_{0,0} = 1$. \\
	
\end{definition}

\subsection{} 
\label{sub:more a}

Let us consider $H_{n,m} \in \CA_{\text{loc}}$ defined for all coprime integers $a,b$ by:
\begin{equation}
\label{eqn:emn}
\sum_{k=0}^{\infty} H_{ka,kb} \cdot x^k = \exp \left[\sum_{k=1}^\infty \frac {P_{ka,kb}}k \cdot x^k \right] 
\end{equation}
In other words, for every fixed pair of coprime integers $a,b$, the elements $H_{ka,kb}$ will be to complete symmetric functions as the elements $P_{ka,kb}$ are to power-sum functions. In the present paper, we will work with the subalgebra:
\begin{equation}
\label{eqn:def a}
\CA_{\text{loc}} \supset \CA = \BZ[q_1^{\pm 1}, q_2^{\pm 1}] \Big \langle H_{n,m}, c_1^{\pm 1}, c_2^{\pm 1} \Big \rangle_{(n,m) \in \BZ^2 \backslash (0,0)} 
\end{equation}
We note a slight abuse in \eqref{eqn:def a}: the notation implies that the structure constants of products of $H_{n,m}$'s lie in $\BZ[q_1^{\pm 1}, q_2^{\pm 1}]$, but this is not quite true. The reason is the presence of denominators in \eqref{eqn:relation 1} and \eqref{eqn:relation 2}. However, in \eqref{eqn:relation 2}, the denominator is canceled by $Q_{n,m}$, which is by definition a multiple of $1-q^{-1}$. In \eqref{eqn:relation 1}, the denominator will be canceled by the numerator in all representations in which:
$$
(c_1,c_2) \ \mapsto \ (q^r,1) \ \text{or} \ (1,q^{-1}) 
$$
which will be the case throughout the present paper. However, for all $r \in \BN$, the following subalgebra of $\CA$ is unambiguously well-defined over $\BZ[q_1^{\pm 1}, q_2^{\pm 1}]$:
$$
\Ar = \BZ[q_1^{\pm 1}, q_2^{\pm 1}] \Big \langle H_{n,m} \Big \rangle_{m > - nr}
$$
The subalgebra $\Ar$ is half of $\A$ with respect to a triangular decomposition. \\

\subsection{} 
\label{sub:action functions} 

The following is obtained by combining the action of \cite{FHHSY} with the explicit formulas obtained in \cite{Shuf} (see Theorem 2.15 of \cite{Ops} for the explicit formulas, although the normalization of \loccit is somewhat different from that of the present paper). \\

\begin{theorem}
\label{thm:action functions}

There is an action $\CA \stackrel{\Psi}\curvearrowright \LA$ given by:
\begin{equation}
\label{eqn:action up 1}
c_1 \mapsto 1, \qquad c_2 \mapsto q^{-1},
\end{equation}
\begin{equation}
\label{eqn:action up 2}
P_{0,m} \ \mapsto p_m , \qquad P_{0,-m} \mapsto -q^m \cdot p_m^\dagger 
\end{equation}
while for all $n > 0$ and $m \in \BZ$, we have:
\begin{equation}
\label{eqn:action up 4}  
H_{n,m} \mapsto \int_{|z_1| \gg ... \gg |z_n|} \frac {\prod_{i=1}^n z_i^{\left \lfloor \frac {mi}n \right \rfloor - \left \lfloor \frac {m(i-1)}n \right \rfloor}}{\prod_{i=1}^{n-1} \left(1 - \frac {z_{i+1}q}{z_i} \right) \prod_{i<j} \zeta \left( \frac {z_j}{z_i} \right)}
\end{equation} 
$$
\exp \left[\sum_{k=1}^\infty \frac {z_1^{-k}+...+z_n^{-k}}k \cdot p_k \right] \exp \left[-\sum_{k=1}^\infty \frac {z_1^k+...+z_n^k}k \cdot p_k^\dagger \right] \prod_{a=1}^n  \frac {dz_a}{2\pi i z_a} 
$$
and:
\begin{equation}
\label{eqn:action up 5}  
H_{-n,m}  \mapsto  \int_{|z_1| \ll ... \ll |z_n|} \frac {(-q)^n \prod_{i=1}^n z_i^{\left \lfloor \frac {mi}n \right \rfloor - \left \lfloor \frac {m(i-1)}n \right \rfloor}}{\prod_{i=1}^{n-1} \left(1 - \frac {z_{i+1}q}{z_i} \right) \prod_{i<j} \zeta \left( \frac {z_j}{z_i} \right)} 
\end{equation} 
$$
\exp \left[-\sum_{k=1}^\infty \frac {z_1^{-k}+...+z_n^{-k}}{k \cdot q^k} \cdot p_k \right] \exp \left[\sum_{k=1}^\infty \frac {z_1^k+...+z_n^k}k \cdot p_k^\dagger \right] \prod_{a=1}^n  \frac {dz_a}{2\pi i z_a}
$$
The integrals go over concentric circles, contained inside each other in the order depicted in the subscript of each integral, and far away from each other relative to the size of the parameters $q_1$ and $q_2$ (which are assumed to be complex numbers). \\

\end{theorem}

\begin{proof} We will sketch the proof, in order to prepare for the analogous argument in Theorem \ref{thm:action functions surface}. There is a well-known triangular decomposition:
$$
\CA = \CA^+ \otimes \CA^0 \otimes \CA^-
$$
where $\CA^\pm$ are the subalgebras of $\CA$ generated by $H_{\pm n,m}$ for $(n,m) \in \BN \times \BZ$, and $\CA^0$ is generated by $P_{0,\pm m}$ and the central elements $c_1,c_2$. The main result of \cite{S} implies that, in order to show that formulas \eqref{eqn:action up 1}--\eqref{eqn:action up 5} yield an action $\CA \curvearrowright \LA$, one needs to prove the following two things: \\

\begin{itemize}[leftmargin=*]

\item Formulas \eqref{eqn:action up 4} and \eqref{eqn:action up 5} induce actions of the subalgebras $\CA^+$ and $\CA^-$ on $\LA$. \\

\item The particular cases of \eqref{eqn:relation 1} and \eqref{eqn:relation 2} when $n,n' \in \{-1,0,1\}$ hold, i.e.:
\begin{align}
&\Big[\Psi(P_{0,\pm m}), \Psi(P_{0,\pm m'})\Big] = 0 \label{eqn:need 1} \\
&\Big[ \Psi(P_{0,m}), \Psi(P_{0,-m'}) \Big] = \delta_{m'}^m m (1-q_1^m)(1-q_2^m)q^m \label{eqn:need 2} \\
&\Big[ \Psi(H_{\pm 1,k}), \Psi(P_{0,\pm m}) \Big]= - (1-q_1^m)(1-q_2^m) \cdot \Psi(H_{\pm 1,k\pm m}) \label{eqn:need 3} \\
&\Big[ \Psi(H_{\pm 1,k}), \Psi(P_{0,\mp m}) \Big]= (1-q_1^m)(1-q_2^m)q^{m\delta_\pm^+} \cdot \Psi(H_{\pm 1,k\mp m}) \label{eqn:need 4} \\
&\Big[ \Psi(H_{1,k}), \Psi(H_{-1,k'}) \Big] = \frac {(1-q_1)(1-q_2)}{q^{-1}-1} \begin{cases} \Psi(A_{k+k'}) &\text{if } k+k' > 0 \\ 
1-q^k &\text{if } k+k' = 0 \\ - q^k \Psi(B_{-k-k'})
&\text{if } k+k' < 0 \end{cases}  \label{eqn:need 5}
\end{align}
for all $m,m' \in \BN$ and $k,k' \in \BZ$, where in the last expression, we write:
\begin{align*} 
&\sum_{m=0}^\infty \frac {A_m}{x^m} = \exp \left[ \sum_{m=1}^\infty \frac {p_m}{m x^m}(1-q^{-m}) \right] \\
&\sum_{m=0}^\infty \frac {B_m}{x^m} = \exp \left[ \sum_{m=1}^\infty \frac {p^\dagger_m}{m x^m}(1-q^m) \right]
\end{align*}

\end{itemize}

\noindent The second bullet is a consequence of straightforward computations using Proposition \ref{prop:classic plane}, which we leave as exercises to the interested reader. As for the first bullet, we note that formula \eqref{eqn:action up 5} reads:
\begin{equation}
\label{eqn:kp}
\Psi(H_{-n,m}) = \int_{|z_1| \ll ... \ll |z_n|} r_{n,m}(z_1,...,z_n) X(z_1,...,z_n) 
\end{equation}
where $r_{n,m}(z_1,...,z_n)$ (resp. $X(z_1,...,z_n)$) is the rational function (resp. the expression) in $z_1,...,z_n$ on the first (resp. second) line of \eqref{eqn:action up 5}. If we assume $q_1$ and $q_2$ to be complex numbers with absolute value greater than 1, then one can move the contours in the integral of \eqref{eqn:kp} to $|z_1|=...=|z_n|$, without picking up any new poles. Once one does this, because $X(z_1,...,z_n)$ is symmetric in $z_1,...,z_n$, then replacing $r_{n,m}$ with its symmetrization only changes the value of the integral by an overall factor of $n!$. Explicitly, this means that \eqref{eqn:kp} is equivalent to: 
\begin{equation}
\label{eqn:jp}
\Psi(H_{-n,m}) = \frac 1{n!} \int_{|z_1| = ... = |z_n|} R_{n,m}(z_1,...,z_n) X(z_1,...,z_n) 
\end{equation}
where $R_{n,m} = \text{Sym } r_{n,m}$. An elementary application of \eqref{eqn:computation 2} shows that:
$$
X(z_1,...,z_n) X(z_{n+1},...,z_{n+n'}) = X(z_1,...,z_{n+n'}) \prod^{1\leq i \leq n}_{n+1 \leq j \leq n+n'} \zeta \left(\frac {z_j}{z_i} \right)^{-1}
$$
Therefore, by applying \eqref{eqn:jp} twice, we obtain:
\begin{multline}
\label{eqn:lp}
\Psi(H_{-n,m}) \Psi(H_{-n',m'}) = \\ = \frac 1{(n+n)!}\int_{|z_1| = ... = |z_{n+n'}|} (R_{n,m} * R_{n',m'})(z_1,...,z_{n+n'})  X(z_1,...,z_{n+n'}) 
\end{multline}
where $R_{n,m} * R_{n',m'}$ denotes the rational function in $z_1,...,z_{n+n'}$ given by:
$$
\frac 1{n! n'!} \cdot \sym \left[R_{n,m}(z_1,...,z_n) R_{n',m'}(z_{n+1},...,z_{n+n'})  \prod^{1\leq i \leq n}_{n+1 \leq j \leq n+n'} \zeta \left(\frac {z_i}{z_j} \right) \right]
$$
The operation $*$ gives rise to an associative product on the vector space $\CS$ of symmetric rational functions with certain poles (\cite{FHHSY}), called the shuffle product. It was shown in \cite{Shuf} that the operation:
$$
(\CA^-,\cdot) \rightarrow (\CS,*) \qquad H_{-n,m} \mapsto R_{n,m}
$$
induces an algebra homomorphism. As we have seen by comparing formulas \eqref{eqn:jp} and \eqref{eqn:lp}, the operation:
$$
(\CS,*) \rightarrow (\text{End}(\LA), \cdot) \qquad R_{n,m} \mapsto \text{RHS of \eqref{eqn:jp}}
$$
is also an algebra homomorphism. Composing the aforementioned homomorphisms implies that formulas \eqref{eqn:jp} give a well-defined action of $\CA^-$ on $\LA$. The fact that formulas \eqref{eqn:action up 4} give rise to a well-defined action of $\CA^+$ on $\LA$ is proved analogously. 
	
\end{proof} 

\subsection{} 
\label{sub:pleth not} 

We will use the symbol $X$ to refer to the totality of the variables $x_1,x_2,...$, and thus we will denote the complete and elementary symmetric functions by:
\begin{align}
&\sum_{k=0}^\infty \frac {h_k}{z^k}  = \wedge^\bullet \left(- \frac Xz \right) \label{eqn:complete} \\
&\sum_{k=0}^\infty \frac {e_k}{(-z)^k} =  \wedge^\bullet \left(\frac Xz \right) \label{eqn:elementary}
\end{align}
where $\wedge^\bullet$ is a multiplicative symbol determined by the property that if a vector space $V$ has torus character $\chi$, then $\wedge^\bullet(\chi)$ denotes the torus character of the total exterior power $\wedge^\bullet(V)$. Elements of $\LA$ will generally be denoted by $f[X]$. We will adopt \underline{plethystic notation}, according to which one defines:
\begin{equation}
\label{eqn:plethysm}
f[X \pm (1-q_1)(1-q_2) z] \in \LA [z]
\end{equation}
to be the image of $f[X]$ under the ring homomorphism $\LA \rightarrow \LA[z]$ that sends:
\begin{equation}
\label{eqn:plethysm 2}
p_n \mapsto p_n \pm (1-q_1^n)(1-q_2^n)z^n
\end{equation}
In other words, one computes the plethysm \eqref{eqn:plethysm} by expanding $f[X]$ in the basis \eqref{eqn:basis}, and then replacing each $p_n$ therein according to \eqref{eqn:plethysm 2}. The reader may find a description of plethysm in the language of equivariant $K$--theory in Proposition \ref{prop:plethystic identity}. The following is a well-known and straightforward exercise: \\

\begin{proposition}
	\label{prop:plethysm}
	
	For any $f[X] \in \LA$ and any variable $z$, we have:
	\begin{equation}
	\label{eqn:pleth}
	f \left[ X \pm \left(1-q_1 \right)\left(1- q_2 \right) z \right] = \exp \left[\pm \sum_{k=1}^\infty \frac {p_k^\dagger z^k}k \right] \cdot f[X]
	\end{equation}
	where $p_k^\dagger$ is the adjoint operator defined in Subsection \ref{sub:a2}. \\ 
	
\end{proposition}

\subsection{}
\label{sub:new plane}

Using \eqref{eqn:complete}, \eqref{eqn:elementary} and \eqref{eqn:pleth}, formulas \eqref{eqn:action up 4}--\eqref{eqn:action up 5} take the form:
$$
\Psi(H_{n,m})(f[X]) = \int_{0, X \prec |z_n| \prec ... \prec |z_1| \prec \infty} \frac {\prod_{i=1}^n z_i^{\left \lfloor \frac {mi}n \right \rfloor - \left \lfloor \frac {m(i-1)}n \right \rfloor}}{\prod_{i=1}^{n-1} \left(1 - \frac {z_{i+1}q}{z_i} \right) \prod_{i<j} \zeta \left( \frac {z_j}{z_i} \right)}  
$$
\begin{equation}
\label{eqn:action left}
\wedge^\bullet \left( - \frac X{z_1} \right) ... \wedge^\bullet \left( - \frac X{z_n} \right) \cdot f \left[ X - \left(1-q_1 \right)\left(1- q_2 \right) \sum_{i=1}^n z_i \right] \prod_{a=1}^n  \frac {dz_a}{2\pi i z_a}
\end{equation}
and:
$$
\Psi(H_{-n,m})(f[X]) = \int_{0, X \prec |z_1| \prec ... \prec |z_n| \prec \infty} \frac {(-q)^n  \prod_{i=1}^n z_i^{\left \lfloor \frac {mi}n \right \rfloor - \left \lfloor \frac {m(i-1)}n \right \rfloor}}{\prod_{i=1}^{n-1} \left(1 - \frac {z_{i+1}q}{z_i} \right) \prod_{i<j} \zeta \left( \frac {z_j}{z_i} \right)}  
$$
\begin{equation}
\label{eqn:action right}
\wedge^\bullet \left( \frac X{z_1 q} \right) ... \wedge^\bullet \left( \frac X{z_n q} \right) \cdot f \left[ X + \left(1-q_1 \right)\left(1- q_2 \right) \sum_{i=1}^n z_i \right] \prod_{a=1}^n  \frac {dz_a}{2\pi i z_a}
\end{equation}
Above, the notation $0, X \prec |z_n| \prec ... \prec |z_1| \prec  \infty$ means that we integrate the variables $z_1,...,z_n$ over concentric circles that go in the prescribed order, and are contained between the poles at $0, x_1,x_2,...$ and the pole at $\infty$. Indeed, the variables $z_i$ must have absolute value larger than the variables $x_1, x_2,...$, in order for us to be able to replace the symbols $p_k$ in \eqref{eqn:action up 4}--\eqref{eqn:action up 5} by $x_1^k+x_2^k+...$. \\

\subsection{} 
\label{sub:torus}

We will work over an algebraically closed field of characteristic 0, henceforth denoted by $\BC$. Fix a line $\infty \subset \BP^2$, and let us write $\BA^2 = \BP^2 \backslash \infty$ for the complement. \\

\begin{definition}

Fix $r \in \BN$. For any $d \geq 0$, consider the moduli space:
\begin{equation}
\label{eqn:framed sheaves}
\CM^f_d = \Big\{ (\CF, \phi), \ \CF \text{ a torsion free sheaf on } \BP^2, \CF|_\infty \stackrel{\phi}\cong \CO_\infty^{\oplus r}, c_2(\CF) = d \Big\}
\end{equation}
It is a smooth quasiprojective algebraic variety of dimension $2rd$. \\

\end{definition}

\noindent An isomorphism $\phi$ as in \eqref{eqn:framed sheaves} is called a framing of the torsion-free sheaf $\CF$, and the pair $(\CF,\phi)$ is called a framed sheaf. We will write:
$$
\CM^f = \bigsqcup_{d = 0}^\infty \CM^f_d 
$$
(the rank $r$ of our sheaves will be fixed throughout the present paper). The torus:
$$
T = \BC^* \times \BC^* \times (\BC^*)^r
$$
acts on $\CM^f$ as follows: the first two factors $\BC^* \times \BC^*$ act on sheaves by their underlying action on the standard coordinate directions of $\BA^2$, while $(\BC^*)^r$ acts by multiplication on the isomorphism $\phi$ in \eqref{eqn:framed sheaves}. Therefore, we may consider:
\begin{equation}
\label{eqn:decomposition}
K_{T}(\CM^f) = \prod_{d = 0}^\infty K_{T}(\CM^f_d) 
\end{equation}
Let $\circ \in \BA^2$ denote the origin, and let us consider the derived restriction:
$$
\xymatrix{\CU_\circ \ar@{.>}[d]\\
\CM^f}
$$
of the universal sheaf $\CU$ on $\CM^f \times \BA^2$ to $\CM^f \times \{\circ\} \cong \CM^f$. \\

\subsection{} 
\label{sub:correspondences} 

We will now define certain operators on $K_T(\CM)$, which were shown in \cite{K-theory} to give rise to the elliptic Hall algebra action that was discovered earlier in \cite{FT,SV 1}. \\

\begin{definition}
\label{def:corr}
	
The following moduli spaces are smooth quasiprojective varieties:
\begin{align*}
&\fZ_1 = \Big\{ (\CF \supset_\circ \CF')\Big\}\\
&\fZ_2^\bullet =  \Big\{ (\CF \supset_\circ \CF' \supset_\circ \CF'') \Big\}
\end{align*}
where $\CF \supset_\circ \CF'$ means that $\CF \supset \CF'$ (as framed sheaves) and the quotient $\CF/\CF'$ is isomorphic to the length 1 coherent sheaf supported at $\circ \in \BA^2$. Consider the maps:
$$
\xymatrix{& \fZ_1 \ar[ld]_{p_-} \ar[rd]^{p_+} & \\ \CM & & \CM} \qquad \qquad \xymatrix{& (\CF \supset_\circ \CF') \ar[ld] \ar[rd] & \\ \CF & & \CF'}
$$
$$
\xymatrix{& \fZ_2^\bullet \ar[ld]_{\pi_-} \ar[rd]^{\pi_+} & \\ \fZ_1 & & \fZ_1} \qquad \xymatrix{& (\CF \supset_\circ \CF' \supset_\circ \CF'') \ar[ld] \ar[rd] & \\ (\CF \supset_\circ \CF') &  & (\CF' \supset_\circ \CF'')}
$$
and the line bundles:
$$
\xymatrix{\CL \ar@{.>}[d] \\ \fZ_1}
\qquad \qquad \qquad \xymatrix{\CF_\circ/\CF'_\circ \ar@{.>}[d] \\ (\CF \supset_\circ \CF')}
$$
$$
\xymatrix{\CL_1,\CL_2 \ar@{.>}[d] \\ \fZ^\bullet_2}
\ \ \qquad \xymatrix{\CF'_\circ/\CF''_\circ, \CF_\circ/\CF'_\circ \ar@{.>}[d] \\ (\CF \supset_\circ \CF' \supset_\circ \CF'')}
$$	
	
\end{definition}

\noindent The smoothness of these moduli spaces is proved by analogy with the corresponding statements in Definition \ref{def:corr surface}. However, all we need at the moment is the structure of $\fZ_1$ and $\fZ_2^\bullet$ as dg schemes, which was developed in \cite{K-theory}. The following is the main result of \loccit (see also \cite{Hecke} for notation closer to ours): \\

\begin{theorem} 
\label{thm:action moduli}	
	
There exists an action $\CA \stackrel{\Phi}\curvearrowright K_T(\CM^f)$ given by:
\begin{equation}
\label{eqn:action right 1}
c_1 \mapsto q^r, \qquad c_2 \mapsto 1,
\end{equation}
\begin{align} 
&P_{0,m} \ \mapsto \text{tensoring with } p_m(\CU_\circ) \label{eqn:action right 2} \\
&P_{0,-m} \mapsto  \text{tensoring with } - q^m \cdot p_m(\CU^\vee_\circ) \label{eqn:action right 3}
\end{align}
\footnote{Above, $p_m(\CU)$ means the $m$--th power sum functor: if $\CU_\circ = \sum_i \pm y_i \in K_T(\CM^f)$, then:
$$
p_m(\CU_\circ) = \sum_i \pm y_i^m \in K_T(\CM^f)
$$} while for all $n > 0$ and $m \in \BZ$, we have:
\begin{equation}
\label{eqn:action right 4}
H_{n,m} \mapsto p_{-*} \Big[ \CL^{d_n} \otimes \pi_{-*} \pi_+^* \Big[ \CL^{d_{n-1}} \otimes ... \otimes \pi_{-*} \pi_+^* \Big[ \CL^{d_1} \otimes p_+^* \Big] ... \Big] \Big] 
\end{equation}
and:
\begin{equation}
\label{eqn:action right 5}
H_{-n,m} \mapsto \left[ \frac {\det \CU_\circ}{(-q)^{r-1}}\right]^n \otimes p_{+*} \Big[ \CL^{d_1-r} \otimes ... \otimes \pi_{+*}\pi_-^* \Big[ \CL^{d_n-r} \otimes p_-^* \Big] ... \Big]
\end{equation}
where $d_i = \left \lfloor \frac {mi}n \right \rfloor - \left \lfloor \frac {m(i-1)}n \right \rfloor$. \\ 

\end{theorem}

\subsection{} 
\label{sub:universal}

Recall the decomposition \eqref{eqn:decomposition}, and consider the class of the structure sheaf:
\begin{equation}
\label{eqn:fundamental}
\b1_d \in K_T(\CM^f_d)
\end{equation}
For any symmetric polynomial $f[X] \in \LA$, we consider the so-called \underline{universal class}:
\begin{equation}
\label{eqn:universal class}
f[\CU_\circ] \in K_T(\CM^f_d)
\end{equation}
by applying the symmetric polynomial $f$ to the Chern roots of the universal sheaf $\CU_\circ$ on $\CM^f_d$. It is well-known that $K_T(\CM^f_d)$ is spanned by universal classes for every $d \geq 0$, i.e. by \eqref{eqn:universal class} as $f[X]$ ranges over $\LA$ (this fact holds for all Nakajima quiver varieties, of which $\CM^f_d$ is an example). Then formulas \eqref{eqn:action right 2} imply that $K_T(\CM^f_d)$ is generated by the operators $P_{0,1},P_{0,2},...$ acting on the class \eqref{eqn:fundamental}, for every $d \geq 0$. This also happens in the case of general surfaces, as we will see in Section \ref{sec:surfaces}. \\

\begin{proposition} 
\label{prop:action universal}
	
(\cite{W surf}) In terms of universal classes, \eqref{eqn:action right 4}--\eqref{eqn:action right 5} read:
$$
\Phi(H_{n,m})(f[\CU_\circ]) = \int_{\CU_\circ \prec |z_n| \prec ... \prec |z_1| \prec 0, \infty} \frac {\prod_{i=1}^n z_i^{\left \lfloor \frac {mi}n \right \rfloor - \left \lfloor \frac {m(i-1)}n \right \rfloor}}{\prod_{i=1}^{n-1} \left(1 - \frac {z_{i+1}q}{z_i} \right) \prod_{i<j} \zeta \left( \frac {z_j}{z_i} \right)}  
$$
\begin{equation}
\label{eqn:action universal 1}
\wedge^\bullet \left(- \frac {\CU_\circ}{z_1} \right) ... \wedge^\bullet \left(- \frac {\CU_\circ}{z_n} \right) \otimes f \left[\CU_\circ-(1-q_1)(1-q_2) \sum_{i=1}^n z_i \right] \prod_{a=1}^n \frac {dz_a}{2\pi i z_a}
\end{equation}
and:
$$
\Phi(H_{-n,m})(f[\CU_\circ]) = \int_{\CU_\circ \prec |z_1| \prec ... \prec |z_n| \prec 0, \infty} \frac {(-q)^{n}\prod_{i=1}^n z_i^{\left \lfloor \frac {mi}n \right \rfloor - \left \lfloor \frac {m(i-1)}n \right \rfloor}}{\prod_{i=1}^{n-1} \left(1 - \frac {z_{i+1}q}{z_i} \right) \prod_{i<j} \zeta \left( \frac {z_j}{z_i} \right)}  
$$
\begin{equation}
\label{eqn:action universal 2}
\wedge^\bullet \left(\frac {\CU_\circ}{z_1q} \right) ... \wedge^\bullet \left(\frac {\CU_\circ}{z_nq} \right) \otimes f \left[\CU_\circ+(1-q_1)(1-q_2) \sum_{i=1}^n z_i \right] \prod_{a=1}^n \frac {dz_a}{2\pi i z_a}
\end{equation}
where $\wedge^\bullet \left(\frac {\CU_\circ}z \right) = \sum_{i=0}^\infty (-z)^{-i} [\wedge^i ( \CU_\circ )]$. \\

\end{proposition}

\noindent Recall from the last paragraph of Subsection \ref{sub:new plane} that the notation $\CU \prec |z_n| \prec ... \prec |z_1| \prec 0, \infty$ means that we integrate the variables $z_1,...,z_n$ over concentric circles that go in the prescribed order, and are contained between the Chern roots of the universal sheaf $\CU_\circ$ and the poles at $0$ and $\infty$. \\

\begin{proof}\emph{of Theorem \ref{thm:main}:} It is easy to see that the operator $\GA$ of \eqref{eqn:gamma} is given by:	
$$
\GA(f[X]) = f[\CU_\circ]
$$
in the notations of Subsections \ref{sub:pleth not} and \ref{sub:universal}, respectively. The fact that $\GA$ commutes with $P_{0,m}$ for any $m > 0$ is an immediate consequence of comparing \eqref{eqn:action up 2} and \eqref{eqn:action right 2}. As for the fact that $\Gamma$ intertwines $\Psi(H_{\pm n, m})$ with $\Phi(H_{\pm n, m})$ for all $n>0$ and $m > \mp nr$, this follows by comparing formulas \eqref{eqn:action left}--\eqref{eqn:action right} with \eqref{eqn:action universal 1}--\eqref{eqn:action universal 2}: either of these formulas involve one and the same integrand, the distinction between them being the location of the contours. Specifically, the contours in \eqref{eqn:action left}--\eqref{eqn:action right} differ from the ones in \eqref{eqn:action universal 1}--\eqref{eqn:action universal 2} only in which side of the contour the pole at 0 lies. The integrals are equal because the integrands are regular at 0 in each variable among $z_1$,...,$z_n$, which is easily seen to be the case for \eqref{eqn:action universal 1}--\eqref{eqn:action universal 2} when $m > \mp nr$. \\

\noindent Concerning the uniqueness statement, let us show that there exists at most a unique:
\begin{equation}
\label{eqn:gamma unique}
\Gamma = \prod_{d=0}^\infty \Gamma_d \quad \text{with} \quad \Gamma_d : \Lambda_{\BA^2,\text{loc}} \rightarrow K_T(\CM^f_d)_{\text{loc}} 
\end{equation}
where:
\begin{align*}
&\Lambda_{\BA^2,\text{loc}} = \LA \bigotimes_{\BZ[q_1^{\pm 1}, q_2^{\pm 1}]} \BQ(q_1,q_2) \\
&K_T(\CM^f_d)_{\text{loc}} = K_T(\CM^f_d) \bigotimes_{K_T(\circ)} \text{Frac}(K_T(\circ))
\end{align*}
such that $\Gamma$ is determined by the facts that $\Gamma_0(1) = \b1_0$ and that $\Gamma$ commutes with the action of $\Ar$, in the sense of diagram \eqref{eqn:main}. The commutativity with the operators $P_{0,m}$ for $m>0$ uniquely determine $\Gamma_0$. Meanwhile, we have the following. \\

\begin{claim}
\label{claim:generate}

Any class $\gamma \in K_T(\CM^f_d)_{\emph{loc}}$ is uniquely determined by the collection: 
$$
\Big\{ \Phi(H_{1,m_1}...H_{1,m_d})(\gamma) \in K_T(\CM^f_0)_{\emph{loc}} \Big\}
$$
as $m_1,...,m_d$ range over the integers $>- r$. \\

\end{claim} 

\noindent The commutativity of diagram \eqref{eqn:main} implies that:
$$
\Phi(H_{1,m_1}...H_{1,m_d})(\Gamma_d(f)) = \Gamma_0(\Psi(H_{1,m_1}...H_{1,m_d})(f))
$$
for all $f \in \Lambda_{\BA^2,\text{loc}}$. Since we have already seen that $\Gamma_0$ is uniquely determined, then Claim \ref{claim:generate} implies that $\Gamma_d(f)$ is uniquely determined, for all $d \geq 0$ and all $f$. \\

\begin{proof}\emph{of Claim \ref{claim:generate}:} Let $\BF = \text{Frac}(K_T(\circ)) = K_T(\CM^f_0)_{\text{loc}}$, where the last equality is due to the fact that $\CM^f_0$ is a point. The Thomason equivariant localization theorem gives us the following isomorphism of $\BF$--vector spaces:
\begin{equation}
\label{eqn:localization}
K_T(\CM_d^f)_{\text{loc}} \cong \bigoplus_{\bla \text{ of size } d} \BF \cdot |\bla \rangle
\end{equation}
where $|\bla \rangle$ denotes the (renormalized) skyscraper sheaf at the $T$--fixed point of $\CM_d^f$ indexed by an $r$--partition $\bla$ of size $d$ (i.e. an $r$--tuple of partitions of total size $d$, we refer to \cite{K-theory, W} for a discussion of the connection between fixed points and $r$--partitions). The classes $|\bla \rangle$ form an orthogonal basis of \eqref{eqn:localization}, with respect to the equivariant Euler characteristic pairing. Since the adjoint of $H_{1,m}$ with respect to this pairing is a multiple of $H_{-1,m+r}$, the claim is equivalent to proving that:
$$
\Big\{ \Phi(H_{-1,m_1}...H_{-1,m_d})(K_T(\CM^f_0)_{\text{loc}}) \Big\}_{m_i > 0} \quad \text{span} \quad K_T(\CM^f_d)_{\text{loc}}
$$
We may prove this claim by induction on $d$, and it suffices to establish that:
\begin{equation}
\label{eqn:generate}
\Big\{ \Phi(H_{-1,m})(K_T(\CM^f_{d-1})_{\text{loc}}) \Big\}_{m > 0} \quad \text{span} \quad K_T(\CM^f_d)_{\text{loc}}
\end{equation}
To prove the claim above, let's consider an $r$--partition $\bmu$ of size $d-1$. We have:
\begin{equation}
\label{eqn:matrix}
\Phi(H_{-1,m})|\bmu\rangle = \sum_{\bla = \bmu+\bsq} \chi_\bsq^m \cdot \tau_\bmu^\bla |\bla \rangle
\end{equation}
where the right-hand side goes over all $r$--partitions obtained by adding a single box $\bsq$ to $\bmu$, and if this box is located at coordinates $(x,y)$ in the $i$--th constituent partition of $\bmu$, its weight is defined by $\chi_\bsq = u_i q_1^x q_2^y$ (see \cite{K-theory, W} for the aforementioned notions and formulas for the coefficients $\tau_\bmu^\bla$ that appear in \eqref{eqn:matrix}, but we remark that they do not depend on $m$). Since there are only finitely many ways to add a single box to the $r$--partition $\bmu$, and all of these boxes have different weights, it is clear that there exists a $\BF$--linear combination $H$ of the operators $H_{-1,1}, H_{-1,2},...$ such that $\Phi(H)|\bmu\rangle = |\bla \rangle$ for any fixed $\bla$. This completes the proof of \eqref{eqn:generate}. 

\end{proof} 

\end{proof}

\section{The case of general surfaces}
\label{sec:surfaces}

\subsection{} 
\label{sub:assumption}

Consider a smooth projective surface $S$ with an ample divisor $H$, and also fix $(r,c_1) \in \BN \times H^2(S, \BZ)$. Consider the moduli space $\CM^s$ of $H$--stable sheaves on the surface $S$ with the numerical invariants $r,c_1$ and any $c_2$. We make the following:
\begin{align}
&\textbf{Assumption A:} \qquad  \gcd(r, c_1 \cdot H) = 1, \ \text{ and} \label{eqn:assumption a} \\
&\textbf{Assumption S:} \ \qquad \text{either } \begin{cases} \CK_S \cong \CO_S \quad \ \ \text{or} \\ c_1(\CK_S) \cdot H < 0 \end{cases} \label{eqn:assumption s}
\end{align}
Assumption A implies that $\CM^s$ is representable, i.e. there exists a universal sheaf:
$$
\xymatrix{\CU \ar@{.>}[d] \\
\CM^s \times S}
$$
We fix a choice of $\CU$ throughout this paper. If we let $\CM^s_d \subset \CM^s$ denote the subspace of sheaves with $c_2 = d$, then we have a disconnected union:
$$
\CM^s = \bigsqcup_{d = \left \lceil \frac {r-1}{2r} c_1^2 \right \rceil}^\infty \CM^s_{d}
$$
where the fact that $d$ is bounded below is a consequence of Bogomolov's inequality. Each $\CM^s_{d}$ is projective (by Assumption A) and smooth (by Assumption S). In the present paper, we will work with the $K$--theory groups:
$$
K(\CM^s) = \prod_{d = \left \lceil \frac {r-1}{2r} c_1^2 \right \rceil}^\infty K(\CM^s_d)
$$
We refer the reader to \cite{Shuf surf, Hecke} for an introduction to basic facts on the moduli space of stable sheaves, as pertains to the present paper. \\

\subsection{} 
\label{sub:weak}

Since $K(\CM^s \times S)  \not \cong K(\CM^s)$, as opposed from the case $S=\BA^2$ studied previously, we must take care what we mean by ``algebras acting on $K$--theory groups". In Definitions \ref{def:weak} and \ref{def:strong}, we let $X$ be any smooth quasiprojective algebraic variety. \\

\begin{definition}
\label{def:weak}
	
A \underline{weak action} $\CA \stackrel{\Phi}\curvearrowright_S K(X)$ is an abelian group homomorphism:
\begin{equation}
\label{eqn:def phi}
\A \xrightarrow{\Phi} \emph{Hom}(K(X), K(X \times S))
\end{equation}
such that: \\

\begin{itemize}[leftmargin=*]

\item $\Phi(1)$ is the standard pull-back map; \\	
	
\item for all $a \in \CA$ and $f \in \BZ[q_1^{\pm 1}, q_2^{\pm 1}]^{\emph{Sym}}$, we require $\Phi(f \cdot a)$ to equal the composition:
\begin{equation}
\label{eqn:eq par}
K(X) \xrightarrow{\Phi(a)} K(X \times S) \xrightarrow{\emph{Id}_{X} \boxtimes  f(q_1,q_2)} K(X \times S)
\end{equation}
where $q_1,q_2$ are identified with the Chern roots of $[\Omega_S^1] \in K(S)$; \\

\item for all $a,b \in \CA$, we require $\Phi(ab)$ to equal the composition:
\begin{equation}
\label{eqn:hom}
K(X) \xrightarrow{\Phi(b)} K(X \times S) \xrightarrow{\Phi(a) \boxtimes \emph{Id}_S} K(X \times S \times S) \xrightarrow{\emph{Id}_{X} \boxtimes \Delta^*} K(X \times S)
\end{equation} 
(where $\Delta : S \hookrightarrow S \times S$ is the diagonal). \\

\end{itemize}

\end{definition}

\noindent We will apply the definition above when $X = \CM^s$ or $X = \bigsqcup_{k=0}^\infty S^k$. \\

\subsection{}
\label{sub:strong} 

Note that in the definition of a weak action, the composition of operators $\Phi(a)$ and $\Phi(b)$ only records what happens on the diagonal of $S \times S$. To understand the behavior off the diagonal, we introduce the following stronger notion. \\

\begin{definition}
\label{def:strong}

A weak action as in Definition \ref{def:weak} is called \underline{strong} if, for all $a,b \in \CA$, we have the following equality of operators $K(X) \rightarrow K(X \times S \times S)$:
\begin{equation}
\label{eqn:comm phi}
[\Phi(a), \Phi(b)] = (\emph{Id}_X \boxtimes \Delta)_*\left[\Phi \left( \frac {[a,b]}{(1-q_1)(1-q_2)} \right) \right]
\end{equation}	
where the left-hand side of \eqref{eqn:comm phi} denotes the difference of the compositions:
\begin{align}
&K(X) \xrightarrow{\Phi^{(2)}(b)} K(X \times S_2) \xrightarrow{\Phi^{(1)}(a) \boxtimes \emph{Id}_{S_2}} K(X \times S_1 \times S_2) \label{eqn:diff 1} \\
&K(X) \xrightarrow{\Phi^{(1)}(a)} K(X \times S_1) \xrightarrow{\Phi^{(2)}(b) \boxtimes \emph{Id}_{S_1}} K(X \times S_1 \times S_2) \label{eqn:diff 2}
\end{align}
(we write $S_i$ instead of $S$ and $K(X) \xrightarrow{\Phi^{(i)}(a)} K(X \times S_i)$ instead of $\Phi(a)$, $\forall i \in \{1,2\}$, in order to better illustrate the two factors of $S$ involved in \eqref{eqn:diff 1}--\eqref{eqn:diff 2}). \\

\end{definition}

\noindent The right-hand side of \eqref{eqn:comm phi} is well-defined, because (see \cite{W surf}) the commutator of any two elements of $\CA$ is a multiple of $(1-q_1)(1-q_2)$. The following operators:
\begin{equation}
\label{eqn:reduced}
[\Phi(a), \Phi(b)]_{\text{red}} = \Phi \left( \frac {[a,b]}{(1-q_1)(1-q_2)} \right)
\end{equation}
which act between $K(X) \rightarrow K(X \times S)$, will be called the \underline{reduced commutators}. \\

\begin{remark}
\label{rem:leibniz jacobi}
	
Consider any $\alpha, \beta, \gamma :  K(X) \rightarrow K(X \times S)$, and any $f \in \BZ[q_1^{\pm 1}, q_2^{\pm 1}]^{\esym}$. Let us define the operators:
$$
f \alpha, \ \alpha \beta : K(X) \rightarrow K(X \times S)
$$
and:
$$
[\alpha, \beta] : K(X) \rightarrow K(X \times S \times S)
$$
by replacing $\Phi(a)$ and $\Phi(b)$ in \eqref{eqn:eq par}, \eqref{eqn:hom}, \eqref{eqn:diff 1}, \eqref{eqn:diff 2} with $\alpha$ and $\beta$. Then we have the following associativity properties:
\begin{equation}
\label{eqn:associativity}
(f \alpha) \beta = f (\alpha \beta), \qquad (\alpha \beta)\gamma = \alpha (\beta \gamma)
\end{equation}
Moreover, assume that the commutator of any two of $\alpha,\beta,\gamma$ is supported on the diagonal $\Delta \subset S \times S$, i.e. we have the following equality $K(X) \rightarrow K(X \times S \times S)$:
\begin{equation}
\label{eqn:red}
[\alpha, \beta] = (\emph{Id}_X \boxtimes \Delta)_*([\alpha, \beta]_{\emph{red}})
\end{equation}
for some operator $[\alpha, \beta]_{\emph{red}} : K(X) \rightarrow K(X \times S)$, and the analogous formulas for the pairs $(\beta, \gamma)$ and $(\alpha, \gamma)$. \footnote{If \eqref{eqn:red} holds for some operator $[\alpha, \beta]_{\text{red}}$, then this operator is unique, due to the fact that the map $(\text{Id}_X \boxtimes \Delta)_*$ has a left inverse given by $(\text{Id}_X \boxtimes \text{proj}_1)_*$} Then the following Leibniz rule holds:
\begin{equation}
\label{eqn:leibniz}
[\alpha \beta, \gamma]_{\emph{red}} = \alpha [\beta,\gamma]_{\emph{red}} + [\alpha, \gamma]_{\emph{red}} \beta
\end{equation}
and the following Jacobi identity holds:
\begin{equation}
\label{eqn:jacobi}
\sum_{\emph{cyclic}} [\alpha,[\beta, \gamma]_{\emph{red}}]_{\emph{red}} = 0
\end{equation}
The claims \eqref{eqn:associativity}, \eqref{eqn:leibniz} and \eqref{eqn:jacobi} are straightforward exercises. \\
	
\end{remark} 

\subsection{} Let us apply Definitions \ref{def:weak} and \ref{def:strong} to the case $X = \CM^f$, $S = \BA^2$ and that of $T$--equivariant $K$--theory. In this case, composing the action map:
$$
\CA \xrightarrow{\Phi} \text{Hom}(K_T(\CM^f), K_T(\CM^f \times \BA^2)) 
$$
with the restriction to the origin $\circ \in \BA^2$ (which is an isomorphism), we obtain:
$$
\A \xrightarrow{\Phi'} \text{End}(K_T(\CM^f)) 
$$
Property \eqref{eqn:eq par} states that $q_1$ and $q_2$ are the equivariant Chern roots of $[\Omega_{\BA^2}^1]$, property \eqref{eqn:hom} states that $\Phi'(ab) = \Phi'(a)\Phi'(b)$, while property \eqref{eqn:comm phi} states that:
$$
[\Phi'(a), \Phi'(b)] = \Phi'([a,b])
$$
the reason being that $\Delta^*\Delta_* = (1-q_1)(1-q_2)$ if $\Delta : \BA^2 \hookrightarrow \BA^2 \times \BA^2$ is the diagonal. The conclusion is that $\Phi'$ yields an honest action of $\CA$ on $K_T(\CM^f)$. \\

\begin{remark} Definitions \ref{def:weak} and \ref{def:strong} are inspired by the Heisenberg algebra action on the cohomology groups of Hilbert schemes that was developed by Grojnowski (\cite{G}) and Nakajima (\cite{Nak}). This construction can be interpreted as ``operators on the cohomology groups of Hilbert schemes of points on a surface $S$, indexed by a cohomology class on $S$". Indeed, if:
\begin{equation}
\label{eqn:phi gamma}
\Phi(a)^{(\gamma)} : K(\CM^s) \rightarrow K(\CM^s)
\end{equation}
denotes the composition:
$$
K(\CM^s) \xrightarrow{\Phi(a)} K(\CM^s \times S) \xrightarrow{\emph{Id}_{\CM^s} \boxtimes \gamma} K(\CM^s \times S) \xrightarrow{\pi_*} K(\CM^s)
$$
for any $\gamma \in K(S)$ (where $\pi:\CM^s \times S \rightarrow \CM^s$ is the projection), then \eqref{eqn:comm phi} reads:
$$
\left[ \Phi(a)^{(\gamma)}, \Phi(b)^{(\delta)} \right] = \Phi \left( \frac {[a,b]}{(1-q_1)(1-q_2)} \right)^{(\gamma \delta)}
$$
for any $\gamma, \delta \in K(S)$. The particular case of the formula above when $a = P_{n,0}$ and $b = P_{n',0}$ yields precisely a Heisenberg algebra action in the sense of \loccit However, since in $K$--theory one does not have a K\"unneth decomposition, the datum of the homomorphism $\Phi(a)$ is stronger than the totality of the endomorphisms \eqref{eqn:phi gamma}. \\

\end{remark}

\subsection{} Let us present the analogues of the correspondences of Subsection \ref{sub:correspondences} with $(\CM^f, \BA^2)$ replaced by $(\CM^s, S)$, and use them to construct an action $\CA \curvearrowright_S K(\CM^s)$. \\

\begin{definition}
\label{def:corr surface}
	
The following moduli spaces are smooth projective varieties:
\begin{align*}
&\fZ_1 = \Big\{ (\CF \supset_x \CF') \text{ for some } x \in S\Big\}\\
&\fZ_2^\bullet =  \Big\{ (\CF \supset_x \CF' \supset_x \CF'') \text{ for some } x \in S \Big\}
\end{align*}
Consider the maps:
$$
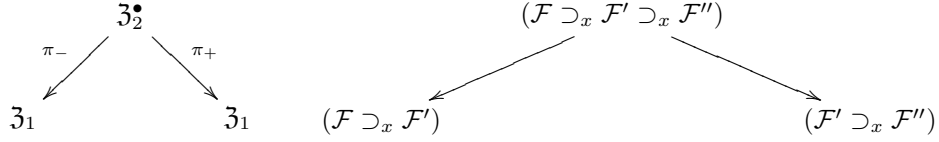
\xymatrix{& \fZ_1 \ar[ld]_{p_-} \ar[d]^{p_S} \ar[rd]^{p_+} & \\ \CM & S & \CM} \qquad \qquad \xymatrix{& (\CF \supset_x \CF') \ar[ld] \ar[d] \ar[rd] & \\ \CF & x & \CF'}
$$
$$
\xymatrix{& \fZ_2^\bullet \ar[ld]_{\pi_-} \ar[rd]^{\pi_+} & \\ \fZ_1 & & \fZ_1} \qquad \xymatrix{& (\CF \supset_x \CF' \supset_x \CF'') \ar[ld] \ar[rd] & \\ (\CF \supset_x \CF') &  & (\CF' \supset_x \CF'')}
$$
and the line bundles:
$$
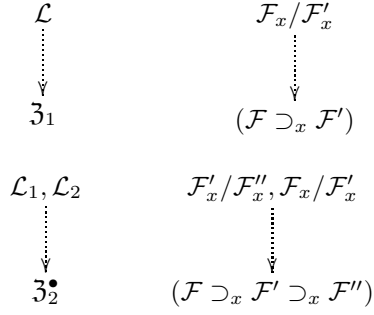
\xymatrix{\CL \ar@{.>}[d] \\ \fZ_1}
\qquad \qquad \qquad \xymatrix{\CF_x/\CF'_x \ar@{.>}[d] \\ (\CF \supset_x \CF')}
$$
$$
\xymatrix{\CL_1,\CL_2 \ar@{.>}[d] \\ \fZ^\bullet_2}
\ \ \qquad \xymatrix{\CF'_x/\CF''_x, \CF_x/\CF'_x \ar@{.>}[d] \\ (\CF \supset_x \CF' \supset_x \CF'')}
$$	
	
\end{definition}

\noindent We refer the reader to \cite{Shuf surf} for the statements pertaining to $\fZ_1$ (although they were known for a long time, see \cite{ES}) and to \cite{W surf} for the statements pertaining to $\fZ_2^\bullet$. \\

\subsection{} The following analogue of Theorem \ref{thm:action moduli} was proved in \cite{Hecke}. \\

\begin{theorem}
\label{thm:action surface}
	
There exists a strong action $\CA \stackrel{\Phi}\curvearrowright_S K(\CM^s)$ given by:
\begin{equation}
\label{eqn:action right surface 1}
c_1 \mapsto q^r, \qquad c_2 \mapsto 1,
\end{equation}
(recall from Definition \ref{def:weak} that $q_1,q_2$ are identified with the Chern roots of $[\Omega_S^1]$, hence $q = q_1q_2$ is identified with the canonical line bundle $[\CK_S]$) and:
\begin{align} 
&P_{0,m} \ \ \mapsto \left[ K(\CM^s) \xrightarrow{\text{pull-back}} K(\CM^s \times S)\xrightarrow{\otimes p_m(\CU)} K(\CM^s \times S) \right] \label{eqn:action right surface 2} \\
&P_{0,-m} \mapsto \left[ K(\CM^s) \xrightarrow{\text{pull-back}} K(\CM^s \times S)\xrightarrow{\otimes (-q^m p_m(\CU^\vee))} K(\CM^s \times S) \right] \label{eqn:action right surface 3}
\end{align}
while for all $n > 0$ and $m \in \BZ$, we have:
\begin{equation}
\label{eqn:action right surface 4}
H_{n,m} \mapsto (p_- \times p_S)_* \Big[ \CL_n^{d_n} \otimes \pi_{-*} \pi_+^* \Big[ \CL_{n-1}^{d_{n-1}} \otimes ... \pi_{-*} \pi_+^* \Big[ \CL_1^{d_1} \otimes p_+^* \Big] ... \Big] \Big] 
\end{equation}
and:
\begin{equation}
\label{eqn:action right surface 5}
H_{-n,m} \mapsto \left[ \frac {\det \CU}{(-q)^{r-1}} \right]^n \otimes (p_+ \times p_S)_* \Big[ \CL_1^{d_1-r} \otimes ... \otimes \pi_{+*} \pi_-^* \Big[ \CL_n^{d_n-r} \otimes p_-^* \Big] ... \Big] 
\end{equation}
where $d_i = \left \lfloor \frac {mi}n \right \rfloor - \left \lfloor \frac {m(i-1)}n \right \rfloor$. \\
		
\end{theorem}

\subsection{} 

The analogue of the ring of symmetric functions for an arbitrary surface is:
$$
\LS = \bigoplus_{k=0}^\infty K_{\fS(k)} (S^k)
$$
where $\fS(k)$ permutes the factors of $S^k$. We will (slightly abusively) write:
\begin{align*} 
&\LSS = \bigoplus_{k=0}^\infty K_{\fS(k)} (S^k \times S) \\
&\LSSS = \bigoplus_{k=0}^\infty K_{\fS(k)} (S^k \times S \times S)
\end{align*}
where $\fS(k)$ does not act on the last factor of $S$ or $S \times S$. There exist analogues of the operators \eqref{eqn:ind res} of parabolic induction and restriction, respectively:
\begin{equation}
\label{eqn:operators}
p_k, p_k^\dagger : \LS \longrightarrow \LSS
\end{equation}
which take the form:
\begin{align} 
&K_{\fS(l)} (S^l) \xrightarrow{p_k} K_{\fS(k+l)} (S^{k+l} \times S) \label{eqn:induction surface} \\
&K_{\fS(k+l)} (S^{k+l}) \xrightarrow{p^\dagger_k} K_{\fS(l)} (S^l \times S) \label{eqn:restriction surface}
\end{align} 
and are explicitly given by:
\begin{align} 
&p_k(M) = \text{Ind}_{\fS(k) \times \fS(l)}^{\fS(k+l)} \Big( \underbrace{ \left[p_k \otimes  \CO_{\Delta_{1...k\bullet}} \right] }_{\text{sheaf on } S^k \times S} \ \boxtimes \underbrace{M}_{\text{sheaf on }S^l} \Big) \label{eqn:ind surf} \\ 
&p_k^\dagger(M) = \text{Hom}_{\fS(k)}\left( p_k,  \text{Res}_{\fS(k) \times \fS(l)}^{\fS(k+l)} (M)\Big|_{\Delta_{1...k \bullet \times S^l}} \right) \label{eqn:res surf}
\end{align}
where $\Delta_{1...k\bullet} \subset S^k \times S$ is the small diagonal. In the right-hand side of \eqref{eqn:res surf}, we implicitly pull-back $M$ from $S^{k+l}$ to $S^{k+l} \times S$, and then restrict it to the diagonal obtained by identifying the first $k$ and the last factor, thus obtaining a sheaf on $S^l \times S$. \\

\begin{proposition}
\label{prop:classic surface}

The operators $p_k$, $p_k^\dagger$ give rise to a strong action (in the sense of Definitions \ref{def:weak} and \ref{def:strong}) of the infinite-dimensional Heisenberg algebra, i.e.:
\begin{equation}
\label{eqn:heis surface}
[p_k^\dagger, p_l] = (\emph{Id}_{\LS} \boxtimes \Delta)_* \left( \rho^*\left[\delta_k^l k \frac {(1-q_1^k)(1-q_2^k)}{(1-q_1)(1-q_2)} \right] \otimes \pi^* \right)
\end{equation}
as well as $[p_k,p_l] = [p_k^\dagger, p_l^\dagger] = 0$, as homomorphisms $\LS \rightarrow \LSSS$, where:
$$
\xymatrix{& {\bigsqcup_{n=0}^{\infty}} S^n \times S \ar[ld]_\pi \ar[rd]^\rho & \\
{\bigsqcup_{n=0}^{\infty}} S^n & & S}
$$
are the standard projections. In the left-hand side of \eqref{eqn:heis surface}, the operator $p_k^\dagger$ (respectively $p_l$) acts only on the first (respectively second) factor of $S \times S$. \\

\end{proposition}

\begin{remark}
	
Proposition \ref{prop:classic surface} actually holds for an arbitrary smooth variety $S$, although we will only need the surface case. If $S$ has dimension $d$, then the constant in the square brackets in the right-hand side of \eqref{eqn:heis surface} must be replaced by:
$$
\delta_k^l k \frac {(1-q_1^k)(1-q_2^k)...(1-q_d^k)}{(1-q_1)(1-q_2)...(1-q_d)}
$$
where $q_1+q_2+...+q_d = [\Omega_S^1] \in K(S)$. The proof below applies to arbitrary $d$. \\

\end{remark}

\begin{proof} Recall that elements of $K(S^n)$ represent vector bundles $M$ on $S^n$. Given any permutation $\sigma = \{1,...,n\} \rightarrow \{a_1,...,a_n\}$, we will use the notation $M_{a_1...a_n}$ for the vector bundle $\sigma^*(M)$ on $S^n$ and the associated $K$--theory class. Similarly, elements of $K_{\fS(n)}(S^n)$ represent $\fS(n)$--equivariant vector bundles, i.e. vector bundles $M$ on $S^n$ endowed with isomorphisms $M \cong \sigma^*(M)$ for all $\sigma \in \fS(n)$ that respect the product of permutations. In this language, formula \eqref{eqn:ind surf} reads:
\begin{equation}
\label{eqn:zero}
p_l(M) = \bigoplus^{(l,n)}_{\text{shuffles}} [p_l \otimes \CO_{\Delta_{a_1...a_l\bullet}}] \boxtimes M_{b_1....b_n} 
\end{equation}
where the right-hand side is a vector bundle on $S^{n+l} \times S$ (the index $\bullet$ represents the last factor of $S$) with the action of $\fS(n+l)$ given by permutation of the indices. The term ``$(l,n)$--shuffles" above and henceforth refers to the set of all partitions:
\begin{equation}
\label{eqn:shuffle}
\{a_1<...<a_l\} \sqcup \{b_1<...<b_n\} =\{1,...,n+l\}
\end{equation}
Iterating \eqref{eqn:zero} twice implies that:
\begin{equation}
\label{eqn:dodi}
p_k p_l(M) = \bigoplus^{(k,l,n)}_{\text{shuffles}} [p_k \otimes \CO_{\Delta_{a_1...a_k\circ}} ] \otimes [p_l \otimes \CO_{\Delta_{b_1...b_l\bullet}}] \boxtimes M_{c_1....c_n} 
\end{equation}
where $(k,l,n)$--shuffles are partitions of $\{1,...,n+k+l\}$ into three sets, of sizes $k$, $l$ and $n$, respectively. The right-hand side of \eqref{eqn:dodi} is a vector bundle on $S^{n+k+l} \times S \times S$, where the latter two factors of $S$ are indexed by the symbols $\circ$ and $\bullet$, respectively. It is clear that the right-hand side of \eqref{eqn:dodi} is symmetric under permuting $k \leftrightarrow l$, if we also permute $\circ \leftrightarrow \bullet$. This implies $[p_k,p_l] = 0$, and the statement that $[p_k^\dagger, p_l^\dagger] = 0$ is analogous. As for the commutator \eqref{eqn:heis surface}, we note that \eqref{eqn:res surf} implies:
\begin{equation}
\label{eqn:unu}
p_k^\dagger p_l(M) = \Hom_{\fS(k)} \left(p_k, \bigoplus^{(l,n)}_{\text{shuffles}} [p_l \otimes \CO_{\Delta_{a_1...a_l\bullet}}] \boxtimes M_{b_1....b_n} \Big|_{\Delta_{1...k\circ} \times S^{n+l-k}} \right)
\end{equation}
as a vector bundle on $S^{n+l-k} \times S \times S$ (the indices $\bullet$ and $\circ$ represent the two latter factors of $S$) with the action of $\fS(n+l-k)$ given by permutation of the indices $>k$. As a virtual representation of $\fS(l)$, the power sum $p_l$ has the property that:
\begin{equation}
\label{eqn:restriction zero}
p_l \Big|_{\fS(i) \times \fS(l-i)} = 0
\end{equation}
for all $i \in \{1,...,l-1\}$. We will call any shuffle as in \eqref{eqn:shuffle} of ``type $i$" if:
$$
\{a_1,...,a_i\} \sqcup \{b_1,...,b_{k-i}\} = \{1,...,k\}
$$
Because of \eqref{eqn:restriction zero}, the only shuffles which contribute non-trivially to \eqref{eqn:unu} are those of type $0$ and type $l$. The shuffles of type $0$ correspond to the case when $\{1,...,k\} \subset \{b_1,...,b_n\}$, and their contribution to \eqref{eqn:unu} may be identified with:
\begin{multline}
\label{eqn:doi}
p_l p_k^\dagger(M) = \bigoplus^{\text{type 0}}_{(l,n)\text{--shuffles}} [p_l \otimes \CO_{\Delta_{a_1...a_l\bullet}}] \boxtimes \\ \boxtimes \Hom_{\fS(k)} \left(p_k, M_{b_1....b_n} \Big|_{\Delta_{1...k\circ} \times S^{n-k}} \right)
\end{multline}
Therefore, the difference between \eqref{eqn:unu} and \eqref{eqn:doi} consists precisely of the sum over type $l$ shuffles, i.e. those such that $\{a_1,...,a_l\} \subset \{1,...,k\}$. However, if $k>l$, then the $\text{Hom}_{\fS(k)}(p_k,...)$ space in \eqref{eqn:unu} vanishes because of \eqref{eqn:restriction zero} for $k \leftrightarrow l$. Therefore, the only shuffle which has a non-zero contribution to the difference of \eqref{eqn:unu} and \eqref{eqn:doi} is the one corresponding to $\{a_1,...,a_l\} = \{1,...,k\}$. We thus conclude that:
\begin{equation}
\label{eqn:trei}
[p_k^\dagger, p_l](M) = M \boxtimes \delta_k^l \Hom_{\fS(k)} \left(p_k, p_k \otimes \CO_{\Delta_{1...k\bullet}} \Big|_{\Delta_{1...k\circ}}\right) 
\end{equation} 
In $K$--theory, the restriction of a regular subvariety (in the situation above, the small diagonal $\Delta_{1...k} : S \hookrightarrow S^k$) to itself is equal to the exterior algebra of the conormal bundle to this subvariety. In the situation of \eqref{eqn:trei}, this leads to:
\begin{equation}
\label{eqn:patru}
[p_k^\dagger, p_l](M) = M \boxtimes \Delta_* \left[ \delta_k^l \Hom_{\fS(k)} \left(p_k, p_k \otimes \wedge^\bullet(N^*_{S|S^k}) \right) \right]  
\end{equation} 
where $\Delta \hookrightarrow S \times S$ is the diagonal that identifies the points $\bullet$ and $\circ$. Recall that: 
$$
\Hom_{\fS(k)}(p_k,V) = \text{Tr}_V(\omega_k)
$$
where $\omega_k \in \fS(k)$ is a $k$--cycle. Therefore, \eqref{eqn:patru} implies \eqref{eqn:heis surface} because of the well-known fact that $\text{Tr}_{p_k}(\omega_k) = k$ and the claim below: \\

\begin{claim}

If $N_{S|S^k}$ denotes the normal bundle of the small diagonal in $S^k$, then: 
$$
\emph{Tr}_{\wedge^\bullet(N^*_{S|S^k})}(\omega_k) = \frac {(1-q_1^k)(1-q_2^k)}{(1-q_1)(1-q_2)} 
$$
where $q_1+q_2 = [\Omega_S^1]$. \\

\end{claim}

\noindent The normal bundle arises from the short exact sequence:
$$
0 \longrightarrow TS \xrightarrow{\text{diag}} (TS)^{\oplus k} \longrightarrow N_{S|S^k} \longrightarrow 0
$$
where the $\fS(k)$--action permutes the factors of $TS$. Therefore, we have:
$$
\wedge^\bullet(N^*_{S|S^k}) = \frac {\wedge^\bullet \left( (\Omega_S^1)^{\oplus k} \right)}{\wedge^\bullet (\Omega_S^1)}
$$
Since the denominator of the expression above is a trivial $\fS(k)$--module with $K$--theory class $(1-q_1)(1-q_2)$, it remains to show that:
\begin{equation}
\label{eqn:cinci}
\text{Tr}_{\wedge^\bullet \left( (T^*S)^{\oplus k} \right)}(\omega_k) = (1-q_1^k)(1-q_2^k)
\end{equation}
By the splitting principle, we can assume $\Omega_S^1 \cong \CL_1 \oplus \CL_2$, where $\CL_1$ and $\CL_2$ are line bundles with $K$--theory classes $q_1$ and $q_2$. If we let $l_a$ denote a local section of the line bundle $\CL_a$, then a basis for the local sections of $\wedge^\bullet(\CL_1^{\oplus k} \oplus \CL_2^{\oplus k})$ consists of:
\begin{equation}
\label{eqn:wedges}
l_1^{(i_1)} \wedge l_1^{(i_2)} \wedge ... \wedge l_1^{(i_a)} \wedge l_2^{(j_1)} \wedge l_2^{(j_2)} \wedge ... \wedge l_2^{(j_b)}
\end{equation}
where $l_a^{(i)}$ denotes the section $l_a$ on the $i$--th copy of $\CL_a$ inside $\CL_a^{\oplus k}$. The cycle $\omega_k$ acts on the basis \eqref{eqn:wedges} by increasing the indices $i_a, j_b \in \BZ/n\BZ$ by 1. Therefore, there are only 4 basis elements which are unchanged by $\omega_k$, namely the cases $(a,b) = (0,0)$, $(k,0)$, $(0,k)$, $(k,k)$ of \eqref{eqn:wedges}. These 4 basis elements contribute precisely $1$, $-q_1^k$, $-q_2^k$, $q^k$ (respectively) to the trace, thus implying \eqref{eqn:cinci}.  
	
\end{proof}

\subsection{} By analogy with Subsection \ref{sub:plethysm}, we have:
\begin{align} 
&\sum_{k=0}^\infty \frac {h_k}{z^k} = \exp \left[ \sum_{k=1}^\infty \frac {p_k}{k z^k} \right] \Big|_\Delta \label{eqn:complete surface} \\ 
&\sum_{k=0}^\infty \frac {e_k}{(-z)^k} = \exp \left[- \sum_{k=1}^\infty \frac {p_k}{k z^k} \right] \Big|_\Delta \label{eqn:elementary surface}
\end{align}
where $h_k,e_k$ are operators defined by analogy with \eqref{eqn:ind surf}, specifically:
\begin{align*}
h_k(M) = \text{Ind}_{\fS(k) \times \fS(l)}^{\fS(k+l)} \Big([\text{triv}_{\fS(k)} \otimes  \CO_{\Delta_{1...k\bullet}}] \boxtimes M \Big) \\
e_k(M) = \text{Ind}_{\fS(k) \times \fS(l)}^{\fS(k+l)} \Big([\text{sign}_{\fS(k)} \otimes  \CO_{\Delta_{1...k\bullet}}] \boxtimes M \Big)
\end{align*} 
for all $M \in K_{\fS(l)}(S^l)$. Moreover, we consider:
\begin{align} 
&\sum_{k=0}^\infty h_k^\dagger z^k = \exp \left[ \sum_{k=1}^\infty \frac {p^\dagger_k z^k}{k} \right] \Big|_\Delta \label{eqn:complete surface adj} \\ 
&\sum_{k=0}^\infty e_k^\dagger (-z)^k = \exp \left[- \sum_{k=1}^\infty \frac {p^\dagger_k z^k}{k} \right] \Big|_\Delta \label{eqn:elementary surface adj}
\end{align}
which by analogy with \eqref{eqn:res surf} satisfy:
\begin{align} 
&h_k^\dagger(M) = \text{Hom}_{\fS(k)}\left(\text{triv}_{\fS(k)}, \text{Res}_{\fS(k) \times \fS(l)}^{\fS(k+l)} (M) \Big|_{\Delta_{1...k \bullet \times S^l}} \right)  \label{eqn:triv surf} \\
&e_k^\dagger(M) = \text{Hom}_{\fS(k)}\left(\text{sign}_{\fS(k)},  \text{Res}_{\fS(k) \times \fS(l)}^{\fS(k+l)} (M) \Big|_{\Delta_{1...k \bullet \times S^l}} \right) \label{eqn:sign surf}
\end{align} 
for all $M \in K_{\fS(k+l)}(S^{k+l})$. \\

\begin{remark}
	
In \cite{K}, a similar result to (a categorification of) Proposition \ref{prop:classic surface} was proved by using certain operators $\LS \rightarrow \LS$ indexed by classes $\gamma \in K(S)$. While similar in overall shape with our:
$$
e_k^{(\gamma)} : \LS \xrightarrow{e_k} \LSS \xrightarrow{\emph{Id}_{\LS} \boxtimes \gamma} \LSS \xrightarrow{\pi_*} \LS
$$
and their adjoints, the operators of \loccit are not linear in $\gamma$. Linearity is necessary in order for operators indexed by $\gamma \in K(S)$ to ``glue" to an operator:
$$
\LS \rightarrow \LSS
$$
which is required for the framework of Definitions \ref{def:weak} and \ref{def:strong}. \\
	
\end{remark}

\subsection{} We have the following global analogue of the rational function \eqref{eqn:zeta}:
\begin{equation}
\label{eqn:zeta s}
\zeta^S(x) = \wedge^\bullet (-x \cdot \CO_\Delta) \in K(S \times S)(x)
\end{equation}
whose restriction to the diagonal is precisely:
$$
\wedge^\bullet \Big(-x \cdot \left([\CO_S]-[\Omega_S^1]+[\CK_S]\right) \Big) = \frac {(1-xq_1)(1-xq_2)}{(1-x)(1-xq)} = \zeta(x)
$$
where $[\Omega_S^1] = q_1+q_2 \in K(S)$. We have the following analogue of Theorem \ref{thm:action functions}: \\

\begin{theorem}
\label{thm:action functions surface} 

There is a strong action $\CA \stackrel{\Psi}\curvearrowright_S \LS$ given by:
\begin{equation}
\label{eqn:action up 1 surface}
c_1 \mapsto 1, \qquad c_2 \mapsto q^{-1},
\end{equation}
\begin{equation}
\label{eqn:action up 2 surface}
P_{0,m} \ \mapsto p_m, \qquad P_{0,-m} \mapsto - m q^m \cdot p_m^\dagger, 
\end{equation}
while for all $n > 0$ and $m \in \BZ$, we have:
$$
H_{n,m} \mapsto \int_{|z_1| \gg ... \gg |z_n|} \frac {\prod_{i=1}^n z_i^{\left \lfloor \frac {mi}n \right \rfloor - \left \lfloor \frac {m(i-1)}n \right \rfloor}}{\prod_{i=1}^{n-1} \left(1 - \frac {z_{i+1}q}{z_i} \right) \prod_{i<j} \zeta \left( \frac {z_j}{z_i} \right)}  
$$
\begin{equation}
\label{eqn:action left surface}
\exp \left[\sum_{k=1}^\infty \frac {z_1^{-k}+...+z_n^{-k}}k \cdot p_k \right] \exp \left[-\sum_{k=1}^\infty \frac {z_1^k+...+z_n^k}k \cdot p_k^\dagger \right] \Big|_\Delta \prod_{a=1}^n  \frac {dz_a}{2\pi i z_a} 
\end{equation}
and:
$$
H_{-n,m} \mapsto \int_{|z_1| \ll ... \ll |z_n|} \frac {(-q)^n  \prod_{i=1}^n z_i^{\left \lfloor \frac {mi}n \right \rfloor - \left \lfloor \frac {m(i-1)}n \right \rfloor}}{\prod_{i=1}^{n-1} \left(1 - \frac {z_{i+1}q}{z_i} \right) \prod_{i<j} \zeta \left( \frac {z_j}{z_i} \right)}  
$$
\begin{equation}
\label{eqn:action right surface}
\exp \left[-\sum_{k=1}^\infty \frac {z_1^{-k}+...+z_n^{-k}}{k \cdot q^k} \cdot p_k \right] \exp \left[\sum_{k=1}^\infty \frac {z_1^k+...+z_n^k}k \cdot p_k^\dagger \right] \Big|_\Delta \prod_{a=1}^n  \frac {dz_a}{2\pi i z_a} 
\end{equation}
(see the last sentence of Theorem \ref{thm:action functions} for the meaning of the contours). \\

\end{theorem}

\begin{proof} Let us first show that the assignments \eqref{eqn:action up 1 surface}--\eqref{eqn:action right surface} give rise to a weak action. As shown in \cite{Hecke}, this can be achieved by establishing that the two bullets in the proof of Theorem \ref{thm:action functions} hold. The first bullet is proved almost word-for-word as in the aforementioned Theorem, with the minor modification that the parameters $q_1$ and $q_2$ are now identified with the Chern roots of $\Omega_S^1$. As for the second bullet, it is an immediate consequence of the following analogues of \eqref{eqn:need 1}--\eqref{eqn:need 5}:
\begin{align}
&\Big[ \Psi(P_{0,\pm m}), \Psi(P_{0,\pm m'}) \Big] = 0 \label{eqn:need new 1} \\
&\Big[ \Psi(P_{0,m}), \Psi(P_{0,-m'})\Big] = \Delta_* \left( \rho^*\left[\delta_{m'}^m m  \frac {(1-q_1^m)(1-q_2^m)}{(1-q_1)(1-q_2)} \right] \pi^* \right) \label{eqn:need new 2} \\
&\Big[\Psi(H_{\pm 1,k}), \Psi(P_{0,\pm m})\Big] = \Delta_* \left(- \rho^*\left[ \frac {(1-q_1^m)(1-q_2^m)}{(1-q_1)(1-q_2)} \right] \Psi(H_{\pm 1,k\pm m}) \right) \label{eqn:need new 3} \\
&\Big[\Psi(H_{\pm 1,k}), \Psi(P_{0,\mp m}) \Big] = \Delta_* \left(\rho^*\left[ \frac {(1-q_1^m)(1-q_2^m)q^{m\delta_\pm^+}}{(1-q_1)(1-q_2)} \right] \Psi(H_{\pm 1,k \mp m}) \right) \label{eqn:need new 4} \\
&\Big[\Psi(H_{1,k}), \Psi(H_{-1,k'})\Big] = \Delta_* \left( \frac 1{q^{-1}-1} \begin{cases} \Psi(A_{k+k'}) &\text{if } k+k' > 0 \\ 
\rho^*(1-q^k)\pi^* &\text{if } k+k' = 0 \\ - \Psi(q^k B_{-k-k'})
&\text{if } k+k' < 0 \end{cases} \quad \right) \label{eqn:need new 5}
\end{align}	
(in the context of a weak action, we only need the restriction of formulas \eqref{eqn:need new 1}--\eqref{eqn:need new 5} to the diagonal $\Delta \subset S \times S$) where $\pi : \CM^s \times S \rightarrow \CM^s$ and $\rho : \CM^s \times S \rightarrow S$ are the standard projections, and $A_m, B_m : \LS \rightarrow \LSS$ are defined by:
\begin{align*} 
&\sum_{m=0}^\infty \frac {A_m}{x^m} = \exp \left[ \sum_{m=1}^\infty \frac {p_m}{m x^m}(1-q^{-m}) \right] \Big|_\Delta \\
&\sum_{m=0}^\infty \frac {B_m}{x^m} = \exp \left[ \sum_{m=1}^\infty \frac {p^\dagger_m}{m x^m}(1-q^m) \right] \Big|_\Delta
\end{align*}
Formulas \eqref{eqn:need new 1}--\eqref{eqn:need new 5} are equalities of homomorphisms $K(\CM^s) \rightarrow K(\CM^s \times S \times S)$, which are straightforward consequences of Proposition \ref{prop:classic surface} (in fact, the first two of these formulas are trivial). Therefore, let us prove \eqref{eqn:need new 5} as an illustration, and leave the remaining formulas as exercises to the interested reader. We have:
$$
\Psi(H_{1,k}) = \sum^{\lambda,\mu \text{ partitions}}_{|\lambda| - |\mu| = k} \frac {(-1)^{|\mu|}}{z_\lambda z_\mu} p_\lambda p_\mu^\dagger \Big|_\Delta
$$
$$
\Psi(H_{-1,k'}) = \sum^{\lambda',\mu' \text{ partitions}}_{|\lambda'| - |\mu'| = k'} \frac {(-q)^{1-|\lambda'|}}{z_{\lambda'} z_{\mu'}} p_{\lambda'} p_{\mu'}^\dagger \Big|_\Delta
$$
Therefore, we have: 
$$
[\Psi(H_{1,k}),\Psi(H_{-1,k'})]_{\text{red}} = \sum^{\lambda,\mu,\lambda',\mu' \text{ partitions}}_{|\lambda| - |\mu| = k, |\lambda'|-|\mu'|=k'} \frac {(-1)^{|\mu|}(-q)^{1-|\lambda'|}}{z_\lambda z_\mu z_{\lambda'} z_{\mu'}} \left[ p_\lambda p_\mu^\dagger \Big|_\Delta, p_{\lambda'} p_{\mu'}^\dagger \Big|_\Delta \right]_{\text{red}}
$$
Formula \eqref{eqn:leibniz} allows us to compute the reduced commutator in the right-hand side. Specifically, this reduced commutator picks up a contribution from the pairing of any $k \in \lambda$ (respectively $l \in \mu$) and any $k \in \mu'$ (respectively $l \in \lambda'$), and this contribution is a scalar due to \eqref{eqn:heis surface}. Therefore, one can write the right-hand side of the expression above as a linear combination of expressions of the form:
$$
p_{\tilde{\lambda}} p_{\tilde{\mu}}^\dagger \Big|_\Delta 
$$
whose coefficients are symmetric Laurent polynomials in $q_1$ and $q_2$. However, the right-hand side of \eqref{eqn:need new 5} is also a linear combination of expressions of the same form. The fact that the two sides of \eqref{eqn:need new 5} are equal in the case of $S = \BA^2$ (when $q_1,q_2$ are formal parameters) implies that they are equal in the case of an arbitrary surface (when $q_1,q_2$ are specialized to the Chern roots of the cotangent bundle). \\

\noindent Now that we have showed that formulas \eqref{eqn:action up 1 surface}--\eqref{eqn:action right surface} give rise to a weak action $\CA \curvearrowright K(\CM^s)$, let us show that they in fact produce a strong action. This entails proving \eqref{eqn:comm phi} for all $x,y \in \CA$. Because of the Leibniz rule \eqref{eqn:leibniz}, it suffices to consider the case when $x,y$ range among the generators of $\CA$, namely the $H_{n,m}$'s. For illustration, let us start with the case $x = H_{-1,m}$ with $y = H_{-1,m'}$:
\begin{equation}
\label{eqn:twice}
\Phi(H_{-1,m}) = \int (-q) \exp \left[- \frac {p_k}{k z^k q^k} \right] \exp \left[\frac {p_k^\dagger z^k}k  \right] \Big|_\Delta  \frac {dz}{2\pi i z} 
\end{equation}
Below, we will consider two copies $S_1=S_2=S$ of the surface, and write:
$$
p_k^{(i)}, p_k^{\dagger, (i)} , \Phi^{(i)}(H_{-1,m}) : \LS \rightarrow \Lambda_S \times K(S_{i}) \quad \text{and} \quad q^{(i)} = [\CK_{S_i}] \in K(S_i)
$$
for each $i \in \{1,2\}$. By applying relation \eqref{eqn:twice} twice, we have:
\begin{equation}
\label{eqn:two operators}
\left(\Phi^{(1)}(H_{-1,m}) \boxtimes \text{Id}_{S_2} \right) \circ \Phi^{(2)}(H_{-1,m'}) = \int_{|z_1|\ll |z_2|} q^{(1)} q^{(2)} z_1^m z_2^{m'}
\end{equation}
$$ 
\exp \left[- \frac {p^{(1)}_k}{k z_1^kq^{(1)k}} \right] \exp \left[\frac {p_k^{\dagger,(1)} z_1^k}k  \right] \Big|_\Delta \exp \left[- \frac {p^{(2)}_k}{k z_2^kq^{(2)k}} \right] \exp \left[\frac {p_k^{\dagger,(2)} z_2^k}k  \right] \Big|_\Delta  \frac {dz_1}{2\pi i z_1} \frac {dz_2}{2\pi i z_2} 
$$
As a consequence of \eqref{eqn:heis surface}, we have the following analogue of \eqref{eqn:computation 2}:
\begin{multline*} 
\exp\left[\sum_{k=1}^\infty \frac {p^{\dagger,(1)}_k z_1^k}k \right] \exp\left[- \sum_{k=1}^\infty \frac {p_k^{(2)}}{k z_2^k q^{(2)k}} \right] = \\ = \exp\left[- \sum_{k=1}^\infty \frac {p_k^{(2)}}{k z_2^k q^{(2)k}} \right] \exp\left[\sum_{k=1}^\infty \frac {p^{\dagger,(1)}_k z_1^k}k \right]  \zeta^S \left( \frac {z_2}{z_1} \right)^{-1}
\end{multline*} 
as an equality of operators $\text{Hom}(\LS, \LS \times K(S_1 \times S_2))$, where $\zeta^S$ is the rational function \eqref{eqn:zeta s}. Therefore, formula \eqref{eqn:two operators} may be rewritten as:
\begin{equation}
\label{eqn:comp 1}
\left(\Phi^{(1)}(H_{-1,m}) \boxtimes \text{Id}_{S_2} \right) \circ \Phi^{(2)}(H_{-1,m'}) = \int_{|z_1|\ll |z_2|} q^{(1)} q^{(2)} z_1^m z_2^{m'} \zeta^S \left( \frac {z_2}{z_1}  \right)^{-1}
\end{equation}
$$
\exp \left[-\sum_{k=1}^\infty \left( \frac {p^{(1)}_k}{k z_1^kq^{(1)k}} + \frac {p^{(2)}_k}{k z_2^kq^{(2)k}} \right) \right] \exp \left[\sum_{k=1}^\infty \left( \frac {p_k^{\dagger,(1)} z_1^k}k + \frac {p_k^{\dagger,(2)} z_2^k}k \right) \right] \frac {dz_1}{2\pi i z_1} \frac {dz_2}{2\pi i z_2} 
$$
Similarly, we have:
\begin{equation}
\label{eqn:comp 2}
\left(\Phi^{(2)}(H_{-1,m'}) \boxtimes \text{Id}_{S_1} \right) \circ \Phi^{(1)}(H_{-1,m}) = \int_{|z_1|\gg |z_2|} q^{(1)} q^{(2)} z_1^m z_2^{m'} \zeta^S \left( \frac {z_1}{z_2} \right)^{-1}
\end{equation}
$$
\exp \left[-\sum_{k=1}^\infty \left( \frac {p^{(1)}_k}{k z_1^kq^{(1)k}} + \frac {p^{(2)}_k}{k z_2^kq^{(2)k}} \right) \right] \exp \left[\sum_{k=1}^\infty \left( \frac {p_k^{\dagger,(1)} z_1^k}k + \frac {p_k^{\dagger,(2)} z_2^k}k \right) \right] \frac {dz_1}{2\pi i z_1} \frac {dz_2}{2\pi i z_2} 
$$
However, Remark 3.17 of \cite{Shuf surf} gives us the following formula:
\begin{equation}
\label{eqn:zeta expand}
\zeta^S(x)^{-1} = 1 - [\CO_\Delta] \otimes \frac x{(1-xq_1)(1-xq_2)}
\end{equation} 
where $q_1$ and $q_2$ are the Chern roots of $\Omega_S^1$ on the diagonal inside $S \times S$. Expanding formula \eqref{eqn:zeta expand} in either positive or negative powers of $x$ shows that it is always equal to 1 times a multiple of $[\CO_\Delta]$. Therefore, the difference of \eqref{eqn:comp 1} and \eqref{eqn:comp 2} is a multiple of $[\CO_\Delta]$, and we conclude that:
\begin{equation}
\label{eqn:final 1}
\left[\Phi^{(1)}(H_{-1,m}), \Phi^{(2)}(H_{-1,m'}) \right] = (\text{Id}_{\LS} \boxtimes \Delta)_* (A)
\end{equation}
where $A$ is a certain difference of integrals of rational functions in $q_1$ and $q_2$, times the symmetric expression:
$$
\exp \left[-\sum_{k=1}^\infty \left( \frac {p_k}{k z_1^kq^k} + \frac {p_k}{k z_2^kq^k} \right) \right] \exp \left[\sum_{k=1}^\infty \left( \frac {p_k^{\dagger} z_1^k}k + \frac {p_k^{\dagger} z_2^k}k \right) \right] \Big|_\Delta  : \LS\rightarrow \LSS
$$
However, because of Theorem \ref{thm:action functions}, we have:
\begin{equation}
\label{eqn:final 2}
A = \Phi\left( \frac {[H_{-1,m}, H_{-1,m'}]}{(1-q_1)(1-q_2)} \right)
\end{equation}
in the $S = \BA^2$ case (when $q_1,q_2$ are formal parameters). Therefore, formula \eqref{eqn:final 2} also holds in the situation at hand (when $q_1,q_2$ are the Chern roots of $\Omega_S^1$). The generalization of the argument above to any $x = H_{n,m}$ and $y = H_{n',m'}$ where $n$ and $n'$ have the same sign is straightforward, and we refer the interested reader to the proof of ``Conjecture 5.7 subject to Assumption B" from \cite{W}. \\

\noindent Let us now show how to prove \eqref{eqn:comm phi} for $x = H_{n,m}$ and $y = H_{n',m'}$ where $n$ and $n'$ have opposite signs, and we will do so by induction on $|n|+|n'|$. The base cases of the induction are precisely \eqref{eqn:need new 1}--\eqref{eqn:need new 5}, so let us assume without loss of generality that $n \geq 2$. There exist $u,u',v,v'$ with $u+u' = n$, $v+v' = m$, $u,u' > 0$ such that:
$$
[H_{u,v}, H_{u',v'}] = (1-q_1)(1-q_2) \Big( c \cdot H_{n,m} + ... \Big)
$$
(see \cite{BS}) where the ellipsis denotes a sum of products of $H_{u'',v''}$ with $0 < u'' < n$, and $c$ is a product of expressions of the form $1+q+...+q^{d-1}$. Since $u$ and $u'$ have the same sign, the argument in the previous paragraph implies that:
$$
\Psi(H_{n,m}) = c^{-1} \Big[\Psi(H_{u,v}), \Psi(H_{u',v'})\Big]_{\text{red}} - c^{-1} \left( \Psi(...) \Big|_\Delta \right) 
$$
(see \eqref{eqn:reduced} for the definition of the reduced commutator). We note that $c$ is invertible in $K(S)$ because $q-1$ is nilpotent. Therefore, the Leibniz rule \eqref{eqn:leibniz} and the Jacobi identity \eqref{eqn:jacobi} imply that:
$$
\Big[\Psi(H_{n,m}), \Psi(H_{n',m'})\Big]_{\text{red}} = c^{-1} \Big[\Psi(H_{u,v}), \Big[\Psi(H_{u',v'}), \Psi(H_{n',m'})\Big]_{\text{red}}\Big]_{\text{red}} - 
$$
$$
- c^{-1} \Big[\Psi(H_{u',v'}), \Big[\Psi(H_{u,v}), \Psi(H_{n',m'})\Big]_{\text{red}}\Big]_{\text{red}} - c^{-1} \Big[\Psi(...)\Big|_\Delta, \Psi(H_{n',m'})\Big]_{\text{red}} 
$$
By the induction hypothesis, the right-hand side of the expression is equal to:
$$
c^{-1} \Psi \left(\frac {[H_{u,v}, [H_{u',v'},H_{n',m'}] - [H_{u',v'}, [H_{u,v},H_{n',m'}]}{(1-q_1)^2(1-q_2)^2} - \frac {[...,H_{n',m'}]}{(1-q_1)(1-q_2)} \right)
$$
The usual Leibniz rule and Jacobi identity in $\CA$ show that the expression above is
$$
\Psi \left( \frac {[H_{n,m}, H_{n',m'}]}{(1-q_1)(1-q_2)} \right)
$$
as required. 
	
\end{proof} 

\subsection{} Akin with Subsection \ref{sub:pleth not}, we will denote elements of $\LS$ by $f[X]$, and write:
\begin{equation}
\label{eqn:complete surf} 
\sum_{k=0}^\infty \frac {h_k}{z^k} = \wedge^\bullet \left( - \frac Xz \right)
\end{equation}
\begin{equation}
\label{eqn:elementary surf} 
\sum_{k=0}^\infty \frac {e_k}{(-z)^k} = \wedge^\bullet \left( \frac Xz \right)
\end{equation}
The following notion is analogous to Proposition \ref{prop:plethysm}: \\

\begin{definition}
\label{def:plethysm surface}
	
For any $f[X] \in \LS$ and any variable $z$, define:
\begin{equation}
\label{eqn:pleth surf}	
f \left[ X \pm \CO_{\Delta} z \right] = \exp \left[\pm \sum_{k=1}^\infty \frac {p_k^\dagger z^k}k \right] \Big|_\Delta \cdot f[X] 
\end{equation}	
as an element of $\LSS[z]$. \\
	
\end{definition} 

\noindent Using \eqref{eqn:complete surf}, \eqref{eqn:elementary surf}, \eqref{eqn:pleth surf}, we may rewrite formulas \eqref{eqn:action left surface} and \eqref{eqn:action right surface} as:
$$
\Psi(H_{n,m})(f[X]) = \int_{0, X \prec |z_n| \prec ... \prec |z_1| \prec \infty} \frac {\prod_{i=1}^n z_i^{\left \lfloor \frac {mi}n \right \rfloor - \left \lfloor \frac {m(i-1)}n \right \rfloor}}{\prod_{i=1}^{n-1} \left(1 - \frac {z_{i+1}q}{z_i} \right) \prod_{i<j} \zeta \left( \frac {z_j}{z_i} \right)}  
$$
\begin{equation}
\label{eqn:action left surf}
\wedge^\bullet \left( - \frac X{z_1} \right) ... \wedge^\bullet \left( - \frac X{z_n} \right) \cdot f \left[ X - \CO_\Delta \sum_{i=1}^n z_i \right] \prod_{a=1}^n  \frac {dz_a}{2\pi i z_a}
\end{equation}
and:
$$
\Psi(H_{-n,m})(f[X]) = \int_{0, X \prec |z_1| \prec ... \prec |z_n| \prec \infty} \frac {(-q)^n  \prod_{i=1}^n z_i^{\left \lfloor \frac {mi}n \right \rfloor - \left \lfloor \frac {m(i-1)}n \right \rfloor}}{\prod_{i=1}^{n-1} \left(1 - \frac {z_{i+1}q}{z_i} \right) \prod_{i<j} \zeta \left( \frac {z_j}{z_i} \right)}  
$$
\begin{equation}
\label{eqn:action right surf}
\wedge^\bullet \left( \frac X{z_1 q} \right) ... \wedge^\bullet \left( \frac X{z_n q} \right) \cdot f \left[ X + \CO_\Delta \sum_{i=1}^n z_i \right] \prod_{a=1}^n  \frac {dz_a}{2\pi i z_a}
\end{equation}
These formulas are proved by analogy with \eqref{eqn:action left} and \eqref{eqn:action right}. The contours of integration are explained in the last paragraph of Subsection \ref{sub:new plane}. \\

\subsection{}

We will now bridge Theorems \ref{thm:action surface} and \ref{thm:action functions surface}. To any element $f[X] \in K_{\fS(k)}(S^k)$ $\subset\LS$, we may associate \underline{universal classes} on $\CM^s$ via the construction \eqref{eqn:operator def surface}:
\begin{equation}
\label{eqn:universal}
f[X] \stackrel{\GS}\leadsto f[\CU] := \pi_* \Big( \CU_1 \otimes ... \otimes \CU_k \otimes \rho^*(f[X]) \Big)^{\fS(k)}
\end{equation}
where $\CU$ is the universal sheaf on $\CM^s \times S$, $\CU_i$ denotes its pull-back to $\CM^s \times S^k$ via the $i$--th projection map $S^k \rightarrow S$, and $\pi : \CM^s \times S^k \rightarrow \CM^s$, $\rho : \CM^s \times S^k \rightarrow S^k$ are the projection maps. Just like in the case of $\BA^2$, the universal classes generate the $K$--theory groups of moduli spaces of stable sheaves, as the following result shows. \\

\begin{lemma}
\label{lem:generate}

Any element of $K(\CM^s_d)$ is of the form \eqref{eqn:universal}, for some $k \gg d$. \\

\end{lemma}

\begin{proof} Theorem 1 of \cite{M} shows that the class of the diagonal:
$$
\Delta_{\CM^s_d} \hookrightarrow \CM^s_d \times \CM^s_d
$$
can be expressed as a Chern class of the virtual bundle:
$$
\xymatrix{\CE \ar@{.>}[d] & \ \sum_{i=0}^2 (-1)^i\text{Ext}^i(\CF,\CF') \ar@{.>}[d] \ar@{_{(}->}[l] \\
\CM^s_d \times \CM^s_d & \ (\CF,\CF') \ar@{_{(}->}[l]}
$$
Although the result of \loccit is stated in cohomology, it holds at the level of Chow groups $A^*(\CM_d^s)$ with rational coefficients. Also, while this result is stated for a surface with trivial canonical class (the top option in Assumption S), an analogous argument works for a surface with negative canonical class (the bottom option in Assumption S), as shown in \cite{GT} for Hilbert schemes. Based on these facts, it is standard to show that any element of $A^*(\CM_d^s)$ can be written in the form:
$$
\pi_{*} \Big(\text{ch}(\CU_1) \cdot ... \cdot \text{ch}(\CU_k) \cdot \rho^*(g) \Big) \in A^*(\CM_d^s)
$$
for some $k \gg d$ and some $g \in A^*(S^k)$. Since the Chern character $K(\CM_d^s) \rightarrow A^*(\CM_d^s)$ is an isomorphism (over $\BQ$, as a consequence of Assumption S), then a simple application of the Grothendieck-Hirzebruch-Riemann-Roch theorem shows that any element of $K(\CM_d^s)$ can be written as:
$$
\pi_* \Big( \CU_1 \otimes ... \otimes \CU_k \otimes \rho^*(g) \Big)
$$
for some $g \in K(S^k)$. The formula above is equal to \eqref{eqn:universal} for $f[X] = \sum_{\sigma \in \fS(k)} \sigma^*(g)$. 
	
\end{proof}

\begin{remark} It would be very interesting to prove Lemma \ref{lem:generate} without Assumption S (although in this case, one should rather work with a dg scheme model of the moduli space of stable sheaves, instead of the singular scheme $\CM^s$). However, the argument provided above requires Assumption S in several crucial places. \\	
\end{remark}

\subsection{} We will need a geometric incarnation on the plethysm operation \eqref{eqn:pleth surf}: \\

\begin{proposition}
\label{prop:plethystic identity}

For any $f[X] \in K_{\fS(n)}(S^n)$, we have the following identity:
\begin{equation}
\label{eqn:plethystic identity}
f \left[\CU\pm \CO_{\Delta} z \right] = \pi_{*} \Big[ \left(\CU_1 \pm \CO_{\Delta_{1\bullet}} z \right) ... \left(\CU_k \pm \CO_{\Delta_{n\bullet}} z \right) \otimes f[X] \Big]^{\fS(n)}
\end{equation}
as elements of $K(\CM^s \times S)[z]$, where $\pi : \CM^s \times S^n \times S \rightarrow \CM^s \times S$ is the projection map that forgets the middle $n$ factors of $S$, and $\bullet$ indicates the surviving factor of $S$. \\

\end{proposition}

\begin{proof} Let us prove \eqref{eqn:plethystic identity} in the case $\pm = +$, as the case $\pm = -$ is analogous, and so we leave it to the interested reader. The right-hand side of \eqref{eqn:plethystic identity} is:
$$
\sum_{k=0}^n z^k \cdot \left( \sum_{\{a_1,...,a_k\} \subset \{1,...,n\}} \pi_* \left[ \bigotimes_{s \in \{1,...,n\} \backslash \{a_1,...,a_n\}} \CU_s \otimes \CO_{\Delta_{a_1...a_k\bullet}} \otimes f[X] \right] \right)^{\fS(n)} = 
$$
$$
= \sum_{k=0}^n z^k \cdot \left( \pi^{(k)}_*\left[\CU_1 \otimes ... \otimes \CU_{n-k} \otimes f[X] \Big|_{S^{n-k} \times \Delta_{n-k+1...n\bullet}} \right] \right)^{\fS(n-k) \times \fS(k)}
$$
where $\pi^{(k)} : \CM^s \times S^{n-k} \times S \rightarrow \CM^s \times S$ is the projection map that forgets the middle $n-k$ factors of $S$ (the two rows above are equal because of Frobenius reciprocity). By \eqref{eqn:triv surf}, the bottom row of the expression above is equal to:
$$
\sum_{k=0}^n z^k \cdot \pi^{(k)}_*\left[\CU_1 \otimes ... \otimes \CU_{n-k} \otimes h_k^\dagger(f[X]) \right]^{\fS(n-k)} \stackrel{\eqref{eqn:universal}}=
$$
$$
= \sum_{k=0}^n z^k \cdot \GS \left(h_k^\dagger(f[X]) \right) \stackrel{\eqref{eqn:complete surface}}= \GS \left( \exp \left[\sum_{k=1}^\infty \frac {p_k^\dagger z^k}k \right] \Big|_\Delta \cdot f[X] \right) = \GS \left( f \left[X + \CO_{\Delta} z \right]  \right)
$$  
which is equal to the left-hand side of \eqref{eqn:plethystic identity}. 

\end{proof}

\noindent The following is Proposition 5.12 of \cite{W surf} (see also Theorem 3.16 of \cite{W} for the $S = \BA^2$ case). In \loccitt, the objects denoted by $f[...]$ in formulas \eqref{eqn:action universal surface 1} and \eqref{eqn:action universal surface 2} were actually understood to be the right-hand side of \eqref{eqn:plethystic identity}. The fact they match
our current notion of plethysm \eqref{eqn:pleth surf} is the content of Proposition \ref{prop:plethystic identity}. \\

\begin{proposition} 
\label{prop:action universal surface}
	
In terms of universal classes, \eqref{eqn:action right surface 4}--\eqref{eqn:action right surface 5} read:
$$
\Phi(H_{n,m})(f[\CU]) = \int_{\CU \prec |z_n| \prec ... \prec |z_1| \prec 0, \infty} \frac {\prod_{i=1}^n z_i^{\left \lfloor \frac {mi}n \right \rfloor - \left \lfloor \frac {m(i-1)}n \right \rfloor}}{\prod_{i=1}^{n-1} \left(1 - \frac {z_{i+1}q}{z_i} \right) \prod_{i<j} \zeta \left( \frac {z_j}{z_i} \right)}  
$$
\begin{equation}
\label{eqn:action universal surface 1}
\wedge^\bullet \left(- \frac {\CU}{z_1} \right) ... \wedge^\bullet \left(- \frac {\CU}{z_n} \right) \otimes f \left[\CU - \CO_\Delta \sum_{i=1}^n z_i \right] \prod_{a=1}^n \frac {dz_a}{2\pi i z_a}
\end{equation}
and:
$$
\Phi(H_{-n,m})(f[\CU]) = \int_{\CU \prec |z_1| \prec ... \prec |z_n| \prec 0, \infty} \frac {(-q)^{n}\prod_{i=1}^n z_i^{\left \lfloor \frac {mi}n \right \rfloor - \left \lfloor \frac {m(i-1)}n \right \rfloor}}{\prod_{i=1}^{n-1} \left(1 - \frac {z_{i+1}q}{z_i} \right) \prod_{i<j} \zeta \left( \frac {z_j}{z_i} \right)}  
$$
\begin{equation}
\label{eqn:action universal surface 2}
\wedge^\bullet \left(\frac {\CU}{z_1q} \right) ... \wedge^\bullet \left(\frac {\CU}{z_nq} \right) \otimes f \left[\CU + \CO_\Delta \sum_{i=1}^n z_i \right] \prod_{a=1}^n \frac {dz_a}{2\pi i z_a}
\end{equation}
as elements of $K(\CM^s \times S)$, where $q_1 + q_2 = [\Omega_S^1]$ and $q=q_1q_2 = [\CK_S]$. \\
	
\end{proposition}

\begin{proof}\emph{of Theorem \ref{thm:surface}:} Comparing formulas \eqref{eqn:action left surf}--\eqref{eqn:action right surf} with \eqref{eqn:action universal surface 1}--\eqref{eqn:action universal surface 2}, we observe that they are one and the same integral but over slightly different contours. The difference is in which side of the $z_1,...,z_n$ contours the pole at $0$ lies. As we explained in the proof of Theorem \ref{thm:main}, when $m > \pm nr$, the integrand is actually regular at 0, so formulas \eqref{eqn:action left surf}--\eqref{eqn:action right surf} produce the same result as \eqref{eqn:action universal surface 1}--\eqref{eqn:action universal surface 2}. \\
	
\end{proof} 

\subsection{} Consider a weak action $\CA \stackrel{\Phi}\curvearrowright_S K(X)$. Given any element $v \in K(X)$, the \underline{submodule} generated by $v$ will refer to the subset $\CA \cdot v \subset K(X)$ consisting of linear combinations of the following operators applied to $v$:
\begin{multline*} 
K(X) \xrightarrow{\Phi(a_k)} K(X \times S) \xrightarrow{\Phi(a_{k-1}) \boxtimes \text{Id}_S} ... \\ ... \xrightarrow{\Phi(a_1) \boxtimes \text{Id}_{S^{k-1}}} K(X \times S^k) \xrightarrow{\text{Id}_X \boxtimes \gamma} K(X \times S^k) \xrightarrow{\pi_*} K(X)
\end{multline*}
where $k \in \BN$, $a_1,...,a_{k-1},a_k \in \CA$, $\gamma \in K(S^k)$ are arbitrary, and $\pi : X \times S^k \rightarrow X$ denotes the projection. For the action of Theorem \ref{thm:action surface}, we consider the submodules:
\begin{equation}
\label{eqn:submodules}
\CA \cdot \b1_d \subset K(\CM^s)
\end{equation}
where $\b1_d \in K(\CM^s_d)$ denotes the structure sheaf of the subvariety $\CM_d^s \subset \CM^s$. Since the algebra $\CA$ contains the operators \eqref{eqn:action right surface 2} of tensor product with the universal sheaf, then Lemma \ref{lem:generate} implies that:
$$
K(\CM^s_d) \subset \CA \cdot \b1_d
$$
for all $d \geq \left \lceil \frac {r-1}{2r} c_1^2 \right \rceil$. Therefore, we conclude the following: \\

\begin{proposition}
\label{prop:generators}
	
The $\Ar$--module $K(\CM^s)$ is generated by $\{\b1_d\}_{d \geq \left \lceil \frac {r-1}{2r} c_1^2 \right \rceil}$, i.e.
\begin{equation}
\label{eqn:union}
K(\CM^s) = \bigcup_{d = \left \lceil \frac {r-1}{2r} c_1^2 \right \rceil}^\infty \CA \cdot \b1_d
\end{equation}
Moreover, these submodules are contained inside each other, i.e.:
\begin{equation}
\label{eqn:nested}
\CA \cdot \b1_d \supset \CA \cdot \b1_{d-1}
\end{equation}
for all $d$. \\
	
\end{proposition}

\noindent The final claim in Proposition \ref{prop:generators} follows from:
\begin{equation}
\label{eqn:identity}
\Phi(H_{n,m})^{(\circ)} \cdot \b1_d = \begin{cases} \b1_{d-n} &\text{if } m = 0 \\ 0 &\text{if } m \in \{-1,...,-nr+1\} \end{cases} 
\end{equation}
for the skyscraper sheaf $\circ \in K(S)$ at any point on $S$. Formula \eqref{eqn:identity} is an immediate consequence of relation \eqref{eqn:action universal surface 1} for $f = 1$ (see also Proposition 4.15 of \cite{K-theory}). \\

\end{document}